\documentclass[10pt]{amsart}
\usepackage{amsmath, amssymb}
\usepackage{amsfonts}
\usepackage{mathrsfs}
\usepackage[arrow,matrix,curve,cmtip,ps]{xy}
\usepackage{tikz-cd}

\usepackage{amsthm}

\allowdisplaybreaks

\newtheorem{theorem}{Theorem}[section]
\newtheorem{lemma}[theorem]{Lemma}
\newtheorem{proposition}[theorem]{Proposition}
\newtheorem{corollary}[theorem]{Corollary}

\newtheorem*{theorem*}{Theorem}
\theoremstyle{remark}
\newtheorem{remark}[theorem]{Remark}
\newtheorem{definition}[theorem]{Definition}
\newtheorem{example}[theorem]{Example}

\numberwithin{equation}{section}

\newcommand{\Z}{\mathbb{Z}}
\newcommand{\N}{\mathbb{N}}

\newcommand{\C}{\mathbb{C}}
\newcommand{\T}{\mathbb{T}}

\newcommand{\Li}{\mathcal{L}}
\newcommand{\K}{\mathcal{K}}

\newcommand{\clspan}{\overline{\operatorname{span}}}

\newcommand{\OX}{\mathcal{O}_X}

\newcommand{\Q}{\mathcal{Q}}

\newcommand{\im}{\operatorname{im }}

\newcommand{\clspn}{\overline{\textnormal{span}}}
\newcommand{\tn}{\textnormal}

\newcommand{\gae}{\lower 2pt \hbox{$\, \buildrel {\scriptstyle >}\over {\scriptstyle
\sim}\,$}}

\newcommand{\lae}{\lower 2pt \hbox{$\, \buildrel {\scriptstyle <}\over {\scriptstyle
\sim}\,$}}

\newcommand{\MU}[1]{
\setbox0\hbox{$#1$}
\setbox1\hbox{$W$}
\ifdim\wd0>\wd1 #1^{\sim} \else \widetilde{#1} \fi
}

\begin{document}
\title{A Cuntz-Krieger Uniqueness Theorem for Cuntz-Pimsner Algebras}

\author{Menev\c se Ery\"uzl\"u Paulovicks}

\address{Department of Mathematics, University of Colorado Boulder, Boulder, CO 80309-0395 \\USA}
\email{menevse.paulovicks@gmail.com}

\author{Mark Tomforde}

\address{Department of Mathematics \\ University of Colorado \\ Colorado Springs, CO 80918-3733 \\USA}
\email{mark.tomforde@gmail.com}

\thanks{This work was supported by a grant from the Simons Foundation (\#527708 to Mark Tomforde)}


\date{\today}



\begin{abstract}
We introduce a property of $C^*$-correspondences, which we call Condition~(S), to serve as an analogue of Condition~(L) of graphs.  We use Condition~(S) to prove a Cuntz-Krieger Uniqueness Theorem for Cuntz-Pimsner algebras and obtain sufficient conditions for simplicity of Cuntz-Pimsner algebras.  We also prove that if $\Q$ is a topological quiver with no sinks and $X(\Q)$ is the associated $C^*$-correspondence, then $X(\Q)$ satisfies Condition~(S) if and only if $\Q$ satisfies Condition~(L).  Finally, we consider several examples to compare and contrast Condition~(S) with Schweizer's nonperiodic condition.
\end{abstract}

\maketitle

\section{Introduction}

In 1997 Pimsner introduced a method for constructing a $C^*$-algebra $\mathcal{O}_X$ from a $C^*$-correspondence $X$.  These $C^*$-algebras have been named \emph{Cuntz-Pimsner algebras}, and a great deal of research has focused on understanding how the structure of the Cuntz-Pimsner algebra $\mathcal{O}_X$ is determined by properties of the $C^*$-correspondence $X$.   As with any class of $C^*$-algebras, one wishes to understand the ideal structure of Cuntz-Pimsner algebras and, in particular, to characterize when Cuntz-Pimsner algebras are simple.  Partial results for simplicity of Cuntz-Pimsner algebras were obtained early in the subject's history:  Let $X$ be a $C^*$-correspondence over a $C^*$-algebra $A$.  In  \cite[Corollary~4.4]{KPW} Kajiwara, Pinzari, and Watatani provided sufficient conditions for simplicity of $X$ when $A$ is unital and $X$ is full and finite projective, and the left action of $A$ on $X$ is injective and by compact operators.  In addition, in \cite[Theorem~1]{MS} Muhly and Solel established  sufficient conditions for simplicity of $\mathcal{O}_X$ under the hypotheses that $A$ is unital and $X$ is projective and essential with injective left action.  In 2001 Schweizer introduced a property for $C^*$-correspondences that he called \emph{nonperiodic} and proved that if $A$ is unital and $X$ is full with injective left action, then $\mathcal{O}_X$ is simple if and only if $X$ is nonperiodic and $A$ has no nontrivial invariant ideals \cite[Theorem~3.9]{Sch}.  Subsequently, Muhly and Solel wrote in \cite{MS2} that Schweizer's result was ``the last word on simplicity" for Cuntz-Pimsner algebras.  However, as we shall describe in more detail in this paper, when Schweizer's hypotheses on $X$ and $A$ are removed, it is still an open problem to characterize simplicity of Cuntz-Pimsner algebras, and even more significant, we shall exhibit examples showing that when Schweizer's hypotheses are removed, the conditions ``$X$ is nonperiodic and $A$ has no nontrivial invariant ideals" are insufficient to imply that $\mathcal{O}_X$ is simple.  Consequently, throughout the course of this paper we hope to convince the reader that there remain many open questions concerning simplicity of Cuntz-Pimsner algebras, and Schweizer's result (while significant and elegant) is far from the last word on the topic.

In the development of Cuntz-Pimsner algebras, graph $C^*$-algebras have provided useful intuition and guidance for the investigation of more general Cuntz-Pimsner algebras.  There are two fundamental structure theorems in the theory of graph $C^*$-algebras:  the Gauge-Invariant Uniqueness Theroem and the Cuntz-Krieger Uniqueness Theorem.  Katsura has successfully extended the Gauge-Invariant Uniqueness Theorem  to all Cuntz-Pimsner algebras \cite[Theorem~6.4]{Kat} and used this result to classify the gauge-invariant ideals in any Cuntz-Pimsner algebra \cite[Theorem~8.6]{Kat2}.  By contrast, up to this point there has been no extension of the Cuntz-Krieger Uniqueness Theorem to general Cuntz-Pimsner algebras.  

A graph is said to satisfy Condition~(L) if every cycle in the graph has an exit.  The Cuntz-Krieger Uniqueness Theorem for graph $C^*$-algebras states that if a graph $E$ satisfies Condition~(L) and $\phi : C^*(E) \to A$ is a $*$-homomorphism that is nonzero on all vertex projections, then $\phi$ is injective.  The Cuntz-Krieger Uniqueness Theorem is used to obtain a characterization of simplicity for graph $C^*$-algebras, namely $C^*(E)$ is simple if and only if $E$ satisfies Condition~(L) and $E$ contains no nontrivial saturated hereditary subsets of vertices.

It would be advantageous to generalize the Cuntz-Krieger Uniqueness Theorem to all Cuntz-Pimsner algebras.  Among other applications, such a theorem would provide sufficient conditions for simplicity of general Cuntz-Pimsner algebras as an almost immediate consequence.  The main obstacle to developing a generalized Cuntz-Krieger Uniqueness Theorem has been an inability to formulate an appropriate analogue of Condition~(L) for $C^*$-correspondences.  While versions of the Cuntz-Krieger Uniqueness Theorem have been obtained by Katsura for topological graph $C^*$-algebras \cite[Theorem~5.12]{Kat3} and by Muhly and the second author for topological quiver $C^*$-algebras \cite[Theorem~6.16]{MT2}, in these situations the $C^*$-correspondence is built from an object (e.g., graph, topological graph, quiver) providing notions of edges, paths, and cycles for which Condition~(L) can be formulated.  Up to this point, it has been unclear how to translate Condition~(L) to an appropriate property entirely in terms of the $C^*$-correspondence.

Comparing Schweizer's simplicity result (i.e., a Cuntz-Pimsner algebra is simple if and only if the $C^*$-correspondence is nonperiodic and $A$ has no nontrivial invariant ideals) to the simplicity result for graph $C^*$-algebras (i.e., a graph $C^*$-algebra is simple if and only if the graph satisfies Condition~(L) and the graph has no nontrivial saturated hereditary subsets), and using the fact that invariant ideals in $A$ correspond to saturated hereditary subsets in the graph, it is natural to conjecture that nonperiodic is the appropriate analogue of Condition~(L) for $C^*$-correspondences.  However, we shall show in this paper that this is not the case.  In particular, in Section~\ref{examination-sec} we show that even under the hypotheses of Schweizer's simplicity theorem there are graphs that do not satisfy Condition~(L) but whose $C^*$-correspondences are nonperiodic.  Therefore nonperiodic is an insufficient condition for formulating a general Cuntz-Krieger Uniqueness Theorem for Cuntz-Pimsner algebras.

In this paper we make some of the first progress in establishing an analogue of Condition~(L) for $C^*$-correspondences that do not necessarily come from underlying objects such as graphs, topological graphs, or quivers.  In particular, in Section~\ref{CK-sec} we define a property of $C^*$-correspondences that we call Condition~(S).  We use Condition~(S) to prove a generalization of the Cuntz-Krieger Uniqueness Theorem (see Theorem~\ref{CKUT-for-CP-algebras-thm} and Corollary~\ref{CK-for-CP-cor}) for Cuntz-Pimsner algebras of $C^*$-correspondences with injective left action:  If $X$ satisfies Condition~(S) and $\rho : \mathcal{O}_X \to B$ is a $*$-homomorphism with $\rho|_{\pi_A(A)}$ injective, then $\rho$ is injective.   One of the most striking aspects of our Cuntz-Krieger Uniqueness Theorem is that the only hypothesis we impose on the $C^*$-correspondence is the injectivity of the left action (cf.~Remark~\ref{S-implies-injective-rem}).

In Section~\ref{simplicity-sec} we are able to apply our Cuntz-Krieger Uniqueness Theorem to obtain sufficient conditions for simplicity of a Cuntz-Pimsner algebra.  In particular, for a $C^*$-correspondence with injective left action, we show in Theorem~\ref{CP-simplicity-thm} that if $X$ satisfies Condition~(S) and $A$ has no nontrivial invariant ideals, then $\mathcal{O}_X$ is simple.  We are then able to combine our work with the results of Kwa\'sniewski and Meyer \cite{KM} (specifically, we apply \cite[Corollary~9.16]{KM}) to cover the non-injective left action case, and obtain a result giving sufficient conditions for simplicity of the Cuntz-Pimsner algebra of any $C^*$-correspondence.  We point out that our lack of any restrictions on the $C^*$-correspondence is in stark contrast to other simplicity theorems for Cuntz-Pimsner algebras, which often require a plethora of hypotheses on the $C^*$-correspondence.  For example, Schweizer's simplicity theorem, in addition to needing the left action to be injective, also requires the $C^*$-correspondence to be full and the coefficient algebra to be unital.  Likewise, the various simplicity results for Cuntz-Pimsner algebras obtained in \cite{CKO} cover disparate cases under various hypotheses (e.g., \cite[Proposition~9.4]{CKO} requires the $C^*$-correspondence to be a Hilbert bimodule with non-injective left action, \cite[Proposition~9.5]{CKO} requires the $C^*$-correspondence to have injective left action, be nondegenerate, and be full on the right, and \cite[Proposition~9.6]{CKO} requires $\overline{ \langle X, J_X X \rangle}$ to be liminal).  Consequently, our simplicity theorem applies to classes of Cuntz-Pimsner algebras beyond those covered by Schweizer's simplicity theorem and the simplicity results of \cite{CKO} and \cite{KM}.

In Section~\ref{L-iff-S-sec} we prove that if $\Q$ is a topological quiver with no sinks and $X(\Q)$ is the associated $C^*$-correspondence, then $X(\Q)$ satisfies Condition~(S) if and only if $\Q$ satisfies Condition~(L).  (The corresponding results for topological graphs and graphs follow as special cases.)  This gives additional credibility to Condition~(S) as an appropriate analogue of Condition~(L).

In Section~\ref{examination-sec} we examine Schweizer's nonperiodic condition in more detail to understand its relationship with Condition~(S) and also Condition~(L).  We produce examples showing that even for finite graphs with no sinks and no sources, nonperiodic is distinct from Condition~(L).  We also provide examples showing that if any of the hypotheses of Schweizer's theorem are removed, the conclusion of the theorem need not hold.  This shows that a condition other than ``nonperiodic" is needed to characterize simplicity of Cuntz-Pimsner algebras when any of the hypotheses of Schweizer's theorem are removed.  The conclusions we draw from our examples are summarized in Remark~\ref{conclusions-comparison-rem}.  We then end the paper with two open questions related to Condition~(S), outlining what remains to be accomplished to  develop the precisely correct analogue of Condition~(L) for all $C^*$-correspondences.

It should also be noted that our Cuntz-Krieger Uniqueness Theorem is related to the results of \cite{COP} and \cite{CKO}.  A Cuntz-Pimsner algebra is said to have the \emph{ideal intersection property} when every nonzero ideal in the Cuntz-Pimsner algebra has nonzero intersection with the coefficient algebra.  Establishing a Cuntz-Krieger Uniqueness Theorem is tantamount to characterizing the ideal intersection property for Cuntz-Pimsner algebras (i.e., for a $*$-homomorphism $\rho : \mathcal{O}_X \to B$ one applies the ideal intersection property to $\ker \rho$).  In \cite{COP} the authors introduce a notion of Condition~(L) for algebraic $R$-systems and show in \cite[Theorem~4.2]{COP} that their Condition~(L) is equivalent to the ideal intersection property for the associated algebraic Cuntz-Pimsner ring.

In \cite[Definition~8.1]{CKO} the authors introduce a property for $C^*$-correspondences they call \emph{$J_X$-acyclic}.  The $J_X$-acyclicity condition may be viewed as both an analytic analogue of Condition~(L) from \cite[Theorem~4.2]{COP} as well as a generalization of Schweizer's nonperiodic condition that accounts for the presence of nonzero ideals in $J_X$.  In \cite[Definition~8.3]{CKO} the authors prove that $J_X$-acyclicity of a $C^*$-correspondence is a necessary condition for the associated Cuntz-Pimsner algebra to have the ideal intersection property.  (They also prove a partial converse under the somewhat strong hypothesis that $\overline{ \langle X, J_X X \rangle}$ is liminal.)   Consequently, the results of \cite{CKO} can be viewed as a complementary to the results of this paper.  While \cite[Definition~8.3]{CKO} shows that $J_X$-acyclicity is necessary for the ideal intersection property (and hence necessary for a Cuntz-Krieger Uniqueness Theorem), Theorem~\ref{CKUT-for-CP-algebras-thm} of this paper shows that Condition~(S) is a sufficient condition for a Cuntz-Krieger Uniqueness Theorem.  It remains an open problem to find conditions that are both necessary and sufficient for a Cuntz-Krieger Uniqueness Theorem to hold, but \cite{CKO} combined with the results of paper show that such conditions must lie somewhere between $J_X$-acyclicity and Condition~(S).

\section{Preliminaries}

Let $A$ be a $C^*$-algebra, and let $X$ be a (right) Hilbert $A$-module.  Let $\langle \cdot, \cdot \rangle_A$ denote the $A$-valued inner product on $X$.  We let $\Li(X)$ denote the
$C^*$-algebra of adjointable operators on $X$, and we let $\K (X)$ denote
the closed two-sided ideal of compact operators given by $$\K (X) := \clspan
\{ \Theta_{\xi,\eta} : \xi, \eta \in X \}$$ where $\Theta_{\xi,\eta}$ is
defined by $\Theta_{\xi,\eta} (\zeta) := \xi \langle \eta, \zeta
\rangle_A$.

\begin{definition}
Let $A$ be a $C^*$-algebra.  A \emph{$C^*$-correspondence} over $A$ is a (right)
Hilbert $A$-module $X$ together with a $*$-homomorphism $\phi_X : A \to
\Li(X)$.  We consider $\phi_X$ as providing a left action of $A$ on $X$ by 
$a \cdot x := \phi_X(a) x$.  
\end{definition}

\begin{remark}
When we say that $X$ is a $C^*$-correspondence over a $C^*$-algebra $A$, note that we implicitly have both an $A$-valued inner product $\langle \cdot, \cdot \rangle_A$ and a left action $\phi_X : A \to \Li (X)$, even if $\langle \cdot, \cdot \rangle_A$ and $\phi_X$ are not explicitly specified.  When no confusion arises (e.g., when we have a single $C^*$-algebra $A$ and a single $C^*$-correspondence $X$) we shall often drop the subscripts on both the inner product and the left action and simply write $\langle \cdot , \cdot \rangle := \langle \cdot , \cdot \rangle_A$ and $\phi := \phi_X$ for ease of notation.
\end{remark}

\begin{definition} \label{Toep-defn}
If $X$ is a $C^*$-correspondence over a $C^*$-algebra $A$, a \emph{representation} of
$X$ into a $C^*$-algebra $B$ is a pair $(\psi, \pi)$ consisting of a linear map
$\psi : X \to B$ and a $*$-homomorphism $\pi : A \to B$ satisfying 
\begin{enumerate}
\item[(i)] $\psi(\xi)^* \psi(\eta) = \pi(\langle \xi, \eta \rangle_A)$
\item[(ii)] $\psi(\phi(a)\xi) = \pi(a) \psi(\xi)$ 
\item[(iii)] $\psi(\xi \cdot a) = \psi(\xi) \pi(a)$
\end{enumerate}
for all $\xi, \eta \in X$ and $a \in A$.  We often write $(\psi, \pi) : (X,A) \to B$ to denote this representation, and we let $C^*(\psi, \pi)$ denote the $C^*$-subalgebra of $B$ generated by $\psi(X) \cup \pi(A)$. For a representation $(\psi, \pi) : (X,A) \to B$ the relations (i)--(ii) above imply 
$$C^*(\psi, \pi)= \tn{ $\clspn\{\psi^n(\xi)\psi^m(\eta)^*:$ $\xi\in X^{\otimes n}, \eta\in X^{\otimes m},$ $n,m\geq 0$\}}.$$  A representation $(\psi, \pi)$ is said to be \emph{injective} if $\pi$ is
injective.  Note that in this case $\psi$ will be isometric (and hence also injective) since $$\|
\psi(\xi) \|^2 = \| \psi(\xi)^* \psi(\xi) \| = \| \pi (\langle \xi, \xi \rangle_A)
\| = \| \langle \xi, \xi \rangle_A \| = \| \xi \|^2.$$

\end{definition}

\begin{definition}
For a representation $(\psi, \pi) : (X,A) \to B$ there exists a $*$-homomorphism $\pi^{(1)} : \K (X) \to B$ with the property that
$$\pi^{(1)} (\Theta_{\xi,\eta}) = \psi(\xi) \psi(\eta)^*.$$  See \cite[p.~202]{Pim},
\cite[Lemma~2.2]{KPW}, and \cite[Remark~1.7]{FR} for details on the existence
of this $*$-homomorphism.  Also note that if $(\psi,\pi)$ is an injective
 representation, then $\pi^{(1)}$ is injective as well
\cite[Proposition~1.6(2)]{FR}.
\end{definition}

\begin{definition}
For an ideal $I$ in a $C^*$-algebra $A$ we define $$I^\perp := \{ a \in A :
ab=0 \text{ for all } b \in I \}$$ and we refer to $I^\perp$ as the
\emph{annihilator of $I$ in $A$}.  If $X$ is a $C^*$-correspondence over
$A$, we define an ideal $J(X)$ of $A$ by $$J(X) := \phi_X^{-1}(\K(X)).$$  We also
define an ideal $J_X$ of $A$ by $$J_X := J(X) \cap (\ker \phi_X)^\perp.$$  Note
that $J_X = J(X)$ when $\phi_X$ is injective, and that $J_X$ is the maximal
ideal of $J(X)$ on which the restriction of $\phi_X$ is an injection into $\K(X)$.
\end{definition}

\begin{definition}
If $X$ is a $C^*$-correspondence over a $C^*$-algebra $A$, we say that a representation $(\psi, \pi) : (X,A) \to B$ is \emph{covariant} if
$$\pi^{(1)} (\phi(a)) = \pi(a) \qquad \text{ for all $a \in J_X$}.$$  We say
that a covariant representation $(\psi_X, \pi_A) : (X,A) \to B$ is
\emph{universal} if whenever $(\psi, \pi) : (X,A) \to C$ is a covariant representation there exists a $*$-homomorphism $\rho_{(\psi,\pi)}: C^*(\psi_X,\pi_A) \to C$ with the property that $\psi =
\rho_{(\psi,\pi)} \circ \psi_X$ and $\pi = \rho_{(\psi,\pi)} \circ \pi_A$.
\end{definition}

\begin{definition}
If $X$ is a $C^*$-correspondence over a $C^*$-algebra $A$, the \emph{Cuntz-Pimsner algebra}, which we shall
denote $\mathcal{O}_X$, is the $C^*$-algebra $C^*(\psi_X,\pi_A)$ where
$(\psi_X, \pi_A)$ is a universal covariant representation of $X$.  The existence of $\mathcal{O}_X$ is established  in \cite[Proposition~1.3]{FMR}, and the universal property ensures $\mathcal{O}_X$ is unique up to isomorphism.
\end{definition}

\begin{remark}
The universal property of a Cuntz-Pimsner algebra ensures the existence of a canonical circle action, which is known as the \emph{gauge action}.  Let $X$ be a $C^*$-correspondence over a $C^*$-algebra $A$. If $\mathcal{O}_X$ is the associated Cuntz-Pimsner algebra and $(\psi_X, \pi_A) : (X,A) \to \mathcal{O}_X$ is a universal covariant representation, then for any $z \in \mathbb{T}$ we have that
$(z\psi_X,\pi_A)$ is also a universal covariant representation. Hence by the universal property, there exists a homomorphism $\gamma_z: \mathcal{O}_X \rightarrow \mathcal{O}_X$ such that
$\gamma_z(\pi_A(a))=\pi_A(a)$ for all $a \in A$ and $\gamma_z(\psi_X(\xi))=z\psi_X(\xi)$
for all $\xi \in X$.  Since $\gamma_{z^{-1}}$ is an inverse for this
homomorphism, we see that $\gamma_z$ is an automorphism. Thus we have an action
$\gamma : \mathbb{T} \rightarrow \operatorname{Aut} \mathcal{O}_X$ with the
property that $\gamma_z (\pi_A(a)) = \pi_A(a)$ for all $a \in A$ and $\gamma_z(\psi_X(\xi)) =
z \psi_X(\xi)$ for all $\xi \in X$. Furthermore, a routine $\epsilon / 3$-argument shows that $\gamma$ is
strongly continuous. We call $\gamma$ the \emph{gauge action on $\mathcal{O}_X$}.
\end{remark}

Because $\T$ is compact, averaging over $\gamma$ with respect to normalized Haar
measure gives a conditional expectation $E$ of $\mathcal{O}_X$ onto the fixed point algebra
$\mathcal{O}_X^\gamma$ by $$E(x) := \int_{\T} \gamma_z (x) \, dz \qquad \text{ for $x \in
\mathcal{O}_X$,}$$ where the integral denotes vector-valued integration in $\mathcal{O}_X$ and $dz$ denotes the normalized Haar measure on $\mathbb{T}$.  The map $E$ is completely positive, has norm $1$, and is faithful in the sense
that $E(a^*a) = 0$ implies $a=0$.

\begin{remark}
Let $X$ be a $C^*$-correspondence over a $C^*$-algebra $A$.  Recalling that there is $*$-homomorphism $\phi_X : A \to \Li (X)$ providing a left action of $A$ on $X$ by $a \cdot x := \phi_X (a) (x)$, we let $X^{\otimes n} := X \otimes_A \ldots \otimes_A X$ denote the $n$-fold $A$-balanced tensor product of $X$ with itself.  We note that $X^{\otimes n}$ is a $C^*$-correspondence over $A$ with right action $(x_1 \otimes \ldots \otimes x_n) \cdot a = x_1 \otimes \ldots \otimes x_n \cdot a$, left action $a \cdot (x_1 \otimes \ldots \otimes x_n) = a \cdot x_1 \otimes \ldots \otimes x_n$, and inner product defined inductively by
$$\langle x_1 \otimes \ldots \otimes x_n, y_1 \otimes \ldots \otimes y_n \rangle := \langle x_2 \otimes \ldots \otimes x_n,  \langle x_1, y_1 \rangle \cdot y_2 \otimes \ldots \otimes y_n \rangle.$$
If $(\psi,\pi) : (X,A) \to B$ is a representation of
$X$ in a $C^*$-algebra $B$, then it is shown in \cite[Lemma~3.6]{FMR} that for
each $n \in \N$ there is a linear map $\psi^{\otimes n} : X^{\otimes n} \to B$ satisfying $$\psi^{\otimes n} (x_1 \otimes \ldots \otimes x_n) = \psi(x_1)
\ldots \psi(x_n) \qquad \text{ for $x_1, \ldots x_n \in X$}.$$  Furthermore,
$(\psi^{\otimes n}, \pi) : (X^{\otimes n}, A) \to B$ is a representation
and there exists a $*$-homomorphism $\pi^{(n)} : \K(X^{\otimes n}) \to B$ such that
$$\pi^{(n)}(\Theta_{x,y}) = \psi^{\otimes n}(x) \psi^{\otimes n}(y)^* \qquad \text{
for $x,y \in X^{\otimes n}$}.$$ 
Suppose $X$ is a $C^*$-correspondence over a $C^*$-algebra $A$, and suppose that the left action $\phi_X: A \rightarrow \Li(X)$ is injective.  If $(\psi ,\pi) : (X,A) \to B$ is a covariant representation, we let $1^k$ denote the identity operator on $X^{\otimes k}$ and
define $$\mathcal{C}_n := A \otimes1^n + \K(X) \otimes 1^{n-1} + \K(X^{\otimes
2}) \otimes 1^{n-2} + \ldots + \K(X^{\otimes n}) \subseteq \Li (X^{\otimes n}).$$ Furthermore, the homomorphism
$\kappa_n^{\psi,\pi} : \mathcal{C}_n \to B$ given by $$ \kappa_n^{\psi,\pi} (k_0
\otimes 1^n + k_1 \otimes 1^{n-1} + \ldots + k_n) := \pi(k_0) + \pi^{(1)}(k_1)
+ \ldots \pi^{(n)}(k_n)$$ is injective whenever $\pi$ is injective
\cite[Proposition~4.6]{FMR}.  It is also a fact that if $\mathcal{C} := \varinjlim
\mathcal{C}_n$ under the isometric homomorphisms $c \in \mathcal{C}_n \mapsto c
\otimes 1 \in \mathcal{C}_{n+1}$, then the homomorphisms $\kappa_n^{\psi,\pi} :
\mathcal{C}_n \to B$ induce a map $\kappa^{\psi,\pi} : \mathcal{C} \to B$, and
\cite[Corollary~4.8]{FMR} shows that $\kappa^{\psi,\pi}$ is injective when $\pi$
is injective.  Furthermore, \cite[Corollary~4.9]{FMR} shows that for a universal covariant representation $(\psi_X, \pi_A) : (X,A) \to \mathcal{O}_X$ we have that $\kappa^{\psi_X,\pi_A} : \mathcal{C} \to \OX$
is an injective map onto the fixed point algebra $\OX^\gamma$. 
\end{remark}

\section{A Cuntz-Krieger Uniqueness Theorem for Cuntz-Pimsner algebras} \label{CK-sec}

\begin{definition} \label{nonreturning-def}
Let $X$ be a $C^*$-correspondence over $A$, and let $(\psi_X, \pi_A) : (X,A) \to \mathcal{O}_X$ be a universal covariant representation of $(X,A)$ into $\mathcal{O}_X$.   We define a vector $\zeta \in X^{\otimes m}$ with $m \in \N$ to be \emph{nonreturning} if whenever $n \in \{ 1, \ldots, m-1 \}$ and $\xi \in X^{\otimes n}$, then
$$ \psi_X^{\otimes m} (\zeta)^* \psi_X^{\otimes n} (\xi) \psi_X^{\otimes m} (\zeta) = 0.$$
\end{definition}

\begin{remark}
The properties and nomenclature in our definition of a nonreturning vector are inspired by the Katsura's notion of a nonreturning path in a topological graph (see \cite[Definition~5.5]{Kat3}).  Katsura shows in \cite[Lemma~5.6]{Kat3} and \cite[Lemma~5.7]{Kat3} that a nonreturning path in a topological graph can be used to construct a compactly supported continuous function that is a nonreturning vector in the sense of our Definition~\ref{nonreturning-def}. (Muhly and the second author have also shown in \cite[\S6]{MT} that Katsura's results on nonreturning paths can be generalized to topological quivers.)
\end{remark}

\begin{lemma}
Let $X$ be a $C^*$-correspondence over $A$ and let $\zeta \in X^{\otimes m}$ be a nonreturning vector.  If $(\psi, \pi) : (X,A) \to B$ is any covariant representation, then for any $n \in \{ 1, \ldots, m-1 \}$ and $\xi \in X^{\otimes n}$, we have
$$\psi^{\otimes m} (\zeta)^* \psi^{\otimes n}(\xi) \psi^{\otimes m} (\zeta) = 0.$$
\end{lemma}

\begin{proof}
Let $(\psi_X, \pi_A) : (X,A) \to \mathcal{O}_X$ be a universal covariant  representation of $(X,A)$ into $\mathcal{O}_X$.  If $\rho := \rho_{(\psi, \pi)} : \mathcal{O}_X \to B$ denotes the $*$-homomorphism induced by $(\psi, \pi)$, then 
\begin{align*}
\psi^{\otimes m} (\zeta)^* \psi^{\otimes n}(\xi) \psi^{\otimes m} (\zeta) &= (\rho \circ \psi_X^{\otimes m}) (\zeta)^* (\rho \circ \psi_X^{\otimes n}) (\xi) (\rho \circ \psi_X^{\otimes m}) (\zeta) \\
&= \rho( \psi_X^{\otimes m} (\zeta)^* \psi_X^{\otimes n}(\xi)  \psi_X^{\otimes m} (\zeta) ) = \rho (0) = 0.  \qedhere
\end{align*}
\end{proof}

\noindent \textbf{Condition~(S):} We say that a $C^*$-correspondence $X$ over a $C^*$-algebra $A$ satisfies Condition~(S) if whenever $a \in A$ with $a \geq 0$, then for all $n \in \N$ and for all $\epsilon > 0$ there exists $m > n$ and a nonreturning vector $\zeta \in X^{\otimes m}$ with $\| \zeta \| = 1$ and
$$ \| \langle \zeta, a \zeta \rangle \| > \| a \| - \epsilon.$$

\begin{remark}
The (S) in Condition~(S) is meant to stand for ``simplicity".
\end{remark}

\begin{proposition} \label{Condition-S-implies-injective-prop}
If $X$ is a $C^*$-correspondence over $A$ satisfying Condition~(S), then the left action $\phi_X: A \to \mathcal{L}(X)$ is injective.
\end{proposition}

\begin{proof}
For each $n \in \N$ let $\phi_n : A \to \mathcal{L} (X^{\otimes n})$ denote the $*$-homomorphism defined by $\phi_n(a) (x_1 \otimes \ldots \otimes x_n) := a\cdot x_1 \otimes \ldots x_n$.  Also define $h: \mathcal{L}(X^{\otimes n}) \to \mathcal{L} (X^{\otimes n+1})$ by $h(T) := T \otimes I$.  Then $h$ is a $*$-homomorphism, and hence $\| h(T) \| \leq \| T\|$ for all $T \in \mathcal{L}(X^{\otimes n})$.  Furthermore, for any $a \in A$ and $n \in \N$, we have $h (\phi_n(a)) = \phi_{n+1}(a)$.  Thus $\| \phi_{n+1} (a) \| \leq \| \phi_n(a) \|$.  By induction, we conclude that
$$ \| \phi_n(a) \| \leq \| \phi_1 (a) \| = \| \phi_X (a) \|  \qquad \text{ for all $a \in A$ and $n \in \N$.}$$
For any $b \in A$, let $a := b^* b \geq 0$.  For any $\epsilon >0$, Condition~(S) implies there exists $m \in \N$ and a nonreturning vector $\zeta \in X^{\otimes m}$ with $\| \zeta \| = 1$ and
$$ \| \langle \zeta, \phi_m(a) \zeta \rangle \| > \| a \| - \epsilon.$$
Since $\| \zeta \| = 1$, we have
\begin{align*}
 \| \phi_X(b) \|^2 &= \| \phi_1(b) \|^2 \geq \| \phi_m(b)
\|^2 \geq \| \phi_m(b) \zeta \|^2 = \| \langle \phi_m(b) \zeta, \phi_m(b) \zeta \rangle \| \\
&= \| \langle \zeta, \phi_m(b^*b) \zeta \rangle \| = \| \langle \zeta, \phi_m(a) \zeta \rangle \| > \| a \| - \epsilon = \| b^* b \| - \epsilon = \| b \|^2 - \epsilon.
\end{align*}
Since $\epsilon > 0$ was arbitrary, we deduce $\| \phi_X(b) \|^2 \geq \| b \|^2$ and $\| \phi_X(b) \| \geq \| b \|$.  Moreover, because $\phi$ is a $*$-homomorphism, and hence norm decreasing, we conclude $\| \phi_X(b) \| = \| b \|$, so that $\phi_X$ is isometric and hence also injective.
\end{proof}

\begin{remark} \label{S-implies-injective-rem}
Proposition~\ref{Condition-S-implies-injective-prop} shows that any $C^*$-correspondence satisfying Condition~(S) necessarily has injective left action.  The reader should keep this in mind, and note that in all forthcoming results the hypothesis that a $C^*$-correspondence satisfies Condition~(S) implicitly requires the $C^*$-correspondence to have injective left action.
\end{remark}

\begin{lemma} \label{major-technical-lem}
Let $X$ be a $C^*$-correspondence over a $C^*$-algebra $A$, and suppose that $X$ satisfies Condition~(S). Let $(\psi_X, \pi_A) : (X,A) \to \mathcal{O}_X$ be a universal covariant representation of $X$ in $\mathcal{O}_X$.   Also let $(\psi, \pi) : (X,A) \to B$ be an injective covariant representation of $X$ into a $C^*$-algebra $B$.  If $$x = \sum_{i=1}^N \psi_X^{\otimes n_i} (\xi_i) \psi_X^{\otimes m_i} (\eta_i)^*$$
is positive in $\mathcal{O}_X$, then
$$ y_0 := \sum_{n_i = m_i} \psi^{\otimes n_i} (\xi_i) \psi^{\otimes m_i} (\eta_i)^*$$
is positive in $B$,  and if $\rho := \rho_{(\psi, \pi)} : \mathcal{O}_X \to B$ denotes the $*$-homomorphism induced by $(\psi, \pi)$, then for all $\epsilon > 0$ there exists $a \in A$ and $b \in B$ with $\| a \| = \| y_0 \|$, $\| b \| \leq 1$, and 
$$\| b^*\rho(x)b - \pi (a) \| < \epsilon.$$
\end{lemma}

\begin{proof}
Let $E : \mathcal{O}_X \to \mathcal{O}^\gamma_X$ denote the conditional expectation determined by the gauge action $\gamma$ on $\mathcal{O}_X$ and given by $E(x) := \int_\mathbb{T} \gamma_z(x) \, dz$.  Then 
\begin{align*}
y_0 &=  \sum_{n_i = m_i} \psi^{\otimes n_i} (\xi_i) \psi^{\otimes m_i} (\eta_i)^* = \sum_{n_i = m_i} (\rho \circ \psi_X)^{\otimes n_i} (\xi_i) (\rho \circ \psi_X)^{\otimes m_i} (\eta_i)^* \\
&= \rho \left ( \sum_{n_i = m_i} \psi_X^{\otimes n_i} (\xi_i) \psi_X^{\otimes m_i} (\eta_i)^* \right) = \rho ( E (x)).
\end{align*}
Since $E$ is a conditional expectation, $E$ is completely positive (as well as positive). Likewise, the $*$-homomorphism $\rho$ is positive.  Since $x$ is positive in $\mathcal{O}_X$, it follows that $y_0 = \rho ( E (x))$ is positive in $B$.

Given $\epsilon > 0$, set $n := \max \{ n_1, \ldots n_N, m_1, \ldots, m_N \}$.  By \cite[Proposition~4.6]{FMR} the $*$-homomorphism $\kappa_n^{\psi,\pi} : \mathcal{C}_n \to B$ is injective and hence isometric.  Therefore, since $y_0 =  \sum_{n_i=m_i} \psi^{\otimes n_i} (\xi_i) \psi^{\otimes m_i} (\eta_i)^* \in \im \kappa_n^{\psi,\pi}$, there exists $T \in \mathcal{C}_n \subseteq \mathcal{L} (X^{\otimes n})$ with $\kappa_n^{\psi,\pi} (T) = y_0$ and $\| T \| = \| y_0 \|$.  Moreover, since $y_0$ is positive in $B$, we may choose $T$ positive in $\mathcal{C}_n$: since $y_0 \in  \im \kappa_n^{\psi,\pi} \subseteq B$ is positive in $B$, and hence positive in $\im \kappa_n^{\psi,\pi}$, we may write $y_0 = v^*v$ for $v \in \im \kappa_n^{\psi,\pi}$, choose $S \in \mathcal{C}_n$ with $\kappa_n^{\psi,\pi}(S) = v$, and let $T := S^*S \geq 0$.  Then $\kappa_n^{\psi,\pi}(T) = \kappa_n^{\psi,\pi}(S^*S) = \kappa_n^{\psi,\pi}(S)^*\kappa_n^{\psi,\pi}(S) = v^*v = y_0$.

Since $\| T \| = \| y_0 \|$ and $T$ is positive, there exists $\xi \in X^{\otimes n}$ with $\| \xi \| = 1$ such that $g := \langle \xi, T \xi \rangle$ is a positive element of $A$ with $\| g \| > \| y_0 \| -\epsilon/2$.  When $n_i > m_i$ we have
$$\psi^{\otimes n}(\xi)^* \psi^{\otimes n_i} (\xi_i)  \psi^{\otimes m_i} (\eta_i)^* \psi^{\otimes n}(\xi) =  \psi^{\otimes n'_i}(\xi'_i)$$
for some $\xi'_i \in X^{\otimes n'_i}$, where $n'_i := n_i - m_i$.  Similarly, when $n_i < m_i$ we have
$$\psi^{\otimes n}(\xi)^* \psi^{\otimes n_i} (\xi_i)  \psi^{\otimes m_i} (\eta_i)^* \psi^{\otimes n}(\xi) =  \psi^{\otimes m'_i}(\eta'_i)^*$$
for some $\eta'_i \in X^{\otimes m'_i}$, where $m'_i := m_i - n_i$.   Also, when $n_i = m_i$ we have
\begin{align*}
\psi^{\otimes n}(\xi)^* &\left( \sum_{n_i = m_i} \psi^{\otimes n_i} (\xi_i) \psi^{\otimes m_i} (\eta_i)^* \right) \psi^{\otimes n}(\xi) = \psi^{\otimes n}(\xi)^* y_0 \psi^{\otimes n}(\xi) \\
&= \psi^{\otimes n}(\xi)^* \kappa_n^{\psi, \pi} (T) \psi^{\otimes n}(\xi) = \psi^{\otimes n}(\xi)^* \psi^{\otimes n}(T\xi) = \pi ( \langle \xi, T\xi \rangle) = \pi (g).
\end{align*}
Combining these three equations yields
\begin{align}
\psi^{\otimes n} (\xi)^* \rho(x) \psi^{\otimes n} (\xi) &= \psi^{\otimes n} (\xi)^* \left( \sum_{i=1}^N \psi^{\otimes n_i} (\xi_i) \psi^{\otimes m_i} (\eta_i)^* \right) \psi^{\otimes n} (\xi)  \label{hit-each-side-eq} \\
&= \pi (g) + \sum_{n_i > m_i} \psi^{\otimes n'_i} (\xi'_i) + \sum_{n_i < m_i} \psi^{\otimes m'_i} (\eta'_i)^*. \notag
\end{align}
Since $g \geq 0$, by Condition~(S) there exists $m > n$ and a nonreturning path $\zeta \in X^{\otimes m}$ with $\| \zeta \| = 1$ and $\| \langle \zeta, g \zeta \rangle \| > \| g \| - \epsilon/2$.  

Let $c := \langle \zeta, g \zeta \rangle \in A$.  Then 
$$ \| c \| = \|  \langle \zeta, g \zeta \rangle \| > \| g \| - \epsilon / 2 > (\| y_0 \| - \epsilon/2) - \epsilon/2 = \| y_0 \| - \epsilon.$$
In addition, 
\begin{align*}
\| c \| &= \| \langle \zeta, g \zeta \rangle \| \leq  \| g \|   \| \langle \zeta, \zeta \rangle \| = \| g \| = \| \langle \xi, T \xi \rangle \| \\
&\leq \| T \| \, \| \langle \zeta, \zeta \rangle \| =  \| T \| = \| y_0 \| <  \| y_0 \| + \epsilon.
\end{align*}
These two inequalities show that 
$$ \Big| \| c \| - \| y_0 \| \Big| < \epsilon.$$
Let $b := \psi^{\otimes n} (\xi) \psi^{\otimes m} (\zeta)$.  Then $\| b \| \leq \| \psi^{\otimes n} (\xi) \| \, \| \psi^{\otimes m} (\zeta) \| = \| \xi \| \, \| \zeta \| = 1 \cdot 1 = 1$.  In addition, 
\begin{align*}
b^* \rho(x) b &= \psi^{\otimes m} (\zeta)^* \psi^{\otimes n} (\xi)^*  \rho(x)  \psi^{\otimes n} (\xi) \psi^{\otimes m} (\zeta) \\
&= \psi^{\otimes m} (\zeta)^* \left(  \pi (g) + \sum_{n_i > m_i} \psi^{\otimes n'_i} (\xi'_i) + \sum_{n_i < m_i} \psi^{\otimes m'_i} (\eta'_i)^* \right) \psi^{\otimes m} (\zeta) \ \text{(by \eqref{hit-each-side-eq}) } \\
&= \psi^{\otimes m} (\zeta)^* \pi (g)  \psi^{\otimes m} (\zeta) \qquad \text{(since $\zeta$ is a nonreturning vector)} \\
&= \psi^{\otimes m} (\zeta)^* \psi^{\otimes m} (g \zeta) \\
&= \pi ( \langle \zeta, g \zeta \rangle )\\
&= \pi (c).
\end{align*}
Let $ a:= \frac{ \| y_0 \| }{ \| c \| } c$.  Then $\| a \| = \| y_0 \|$, and 
\begin{align*}
\| b^* \rho(x) b - \pi (a) \| &= \| \pi (c) - \pi (a) \| = \| \pi (c-a) \| \leq \| c - a \| = \left\| c - \frac{ \| y_0 \| }{ \| c \| } c \right\| \\
&=  \left| 1 - \frac{ \| y_0 \| }{ \| c \| } \right| \| c \| = \Big| \| c \| - \| y_0 \| \Big| < \epsilon. \qedhere
\end{align*}
\end{proof}

\begin{lemma} \label{dense-positive-approx-lem}
Let $A_0$ be a dense $*$-subalgebra of a $C^*$-algebra $A$.  If $x \in A$ and $x \geq 0$, then for every $\epsilon >0$ there exists $x_0 \in A_0$ with $x_0 \geq 0$ and $\| x - x_0 \| < \epsilon$.
\end{lemma}

\begin{proof}
If $x = 0$, the result is trivial, so assume $x \neq 0$.  Since $x \geq 0$, we may write $x = v^*v$ for some nonzero $v \in A$.  Using the fact that $A_0$ is dense in $A$, choose $v_0 \in A_0$ with 
$$\| v - v_0 \| \leq \min \left\{ \frac{\epsilon}{3 \| v \|}, \sqrt{\frac{\epsilon}{3}} \right\}.$$
Let $x_0 := v_0^* v_0 \in A_0$.  Then $x_0 \geq 0$, and 
\begin{align*}
\| x - x_0 \| &= \| v^*v - v_0^* v_0 \| \leq \| v^* v - v^* v_0 \| + \| v^* v_0 - v_0^* v_0 \| \\
&\leq \| v^* \| \ \| v - v_0 \| + \| v^* - v_0^* \| \ \| v_0 \|  = \| v \| \ \| v - v_0 \| + \| v - v_0 \| \ \| v_0 \| \\
& \leq  \| v \| \frac{\epsilon}{3 \| v \|} + \| v - v_0 \| \ \| v_0 - v + v \|  \leq \epsilon/3+ \| v - v_0 \| \ \left( \| v_0 - v\|  + \| v \| \right) \\
& \leq \epsilon/3+ \| v - v_0 \|^2   + \| v \| \ \| v - v_0 \| \leq \epsilon / 3 + \left( \sqrt{\frac{\epsilon}{3}} \right)^2 +  \| v \| \frac{\epsilon}{3 \| v \|} \\
& = \epsilon/3 + \epsilon/3 + \epsilon/3 = \epsilon. \qedhere
\end{align*}
\end{proof}

We are now in a position to prove our main result.

\begin{theorem}[Cuntz-Krieger Uniqueness for Cuntz-Pimsner Algebras] \label{CKUT-for-CP-algebras-thm}
Let $X$ be a $C^*$-correspondence over a $C^*$-algebra $A$, and suppose that $X$ satisfies Condition~(S).  Also let $(\psi_X, \pi_A) : (X,A) \to \mathcal{O}_X$ be a universal covariant representation of $X$ into $\mathcal{O}_X$.  If $\rho : \mathcal{O}_X \to B$ is a $*$-homomorphism from $\mathcal{O}_X$ into a $C^*$-algebra $B$ with the property that the restriction $\rho |_{\pi_A(A)}$ is injective, then $\rho$ is injective.
\end{theorem}

\begin{proof}
Let $\gamma$ denote the gauge action on $\mathcal{O}_X$, and let $E : \mathcal{O}_X \to \mathcal{O}_X^\gamma$ denote the conditional expectation given by $$E (x) := \int_{\T} \gamma_z(x) \, dz.$$
We shall show the following statements hold:
\begin{itemize}
\item[(a)] $\rho$ is injective on the fixed point algebra $\mathcal{O}_X^\gamma$, and
\item[(b)] $\| \rho (E(x)) \| \leq \| \rho (x) \|$ whenever $x \in \mathcal{O}_X$ and $x \geq 0$.
\end{itemize}
To see (a), let $\psi := \rho \circ \psi_X$ and $\pi := \rho \circ \pi_A$.   Then $\rho = \rho_{(\psi, \pi)}$; that is, $\rho$ equal to the $*$-homomorphism induced by the representation $(\psi, \pi)$.  For each $n \in \N$, we have 
$\kappa_n^{\psi, \pi} = \rho \circ \kappa_n^{\psi_X, \pi_A}$, and hence $\kappa^{\psi, \pi} = \rho \circ \kappa^{\psi_X, \pi_A}$.  Since $\kappa^{\psi, \pi}$ is injective by \cite[Corollary~4.8]{FMR}, it follows that $\rho$ is injective on $\im \kappa^{\psi_X, \pi_A} = \mathcal{O}_X^\gamma$, so (a) holds.

To see (b), note that since 
$$ \mathcal{O}_X = \clspan \{ \psi_X^{\otimes n} (\xi) \psi_X^{\otimes m} (\eta)^* : \xi \in X^{\otimes n}, \eta \in X^{\otimes m}, \text{ and } n,m \geq 0 \},$$
Lemma~\ref{dense-positive-approx-lem} implies that any positive element in $\mathcal{O}_X$ may be approximated by a positive element from $\operatorname{span} \{ \psi_X^{\otimes n} (\xi) \psi_X^{\otimes m} (\eta)^* : \xi \in X^{\otimes n}, \eta \in X^{\otimes m}, \text{ and } n,m \geq 0 \}$.  Hence it suffices to prove (b) when $x$ is positive and has the form
$$ x = \sum_{i=1}^N \psi_X^{\otimes n_i} (\xi_i) \psi_X^{\otimes m_i} (\eta_i)^*.$$
Keeping
the notation $(\psi, \pi) = (\rho \circ \psi_X, \rho \circ \pi_A)$, if we let 
$$y := \rho(x) = \sum_{i=1}^N  \psi^{\otimes n_i} (\xi_i) \psi^{\otimes m_i}(\eta_i)^*$$ and
$$y_0 := \sum_{n_i = m_i}  \psi^{\otimes n_i} (\xi_i) \psi^{\otimes m_i}(\eta_i)^*,$$
then Lemma~\ref{major-technical-lem} shows that for all $\epsilon > 0$ there exists $a \in A$ and $b \in B$ such that $\| a \| = \| y_0 \|$, $\|b\| \leq 1$, and $\|b^* y
b - \pi(a) \| < \epsilon$.  Thus
\begin{align*}
\| \rho (E(x)) \| & = \left\| \rho \left( \sum_{n_l = m_l}  \psi_X^{\otimes n_l} (\xi_l)
\psi_X^{\otimes m_l}(\eta_l)^* \right ) \right\| = \left\| \sum_{n_l = m_l}  \psi^{\otimes n_l} (\xi_l)
\psi^{\otimes m_l}(\eta_l)^* \right\| \\
&= \| y_0 \| = \| a \| = \| \rho (\pi_A (a))
\| = \| \pi(a) \|  \leq \| b^* y b \| + \| \pi(a) - b^* y b \|  \\
&\leq \| b^* \| \, \|y \| \, \| b \| + \epsilon \leq \| y \| +
\epsilon = \| \rho(x) \| + \epsilon.
\end{align*}
Since this inequality holds for all $\epsilon > 0$, we have that $\| \rho (E(x)) \|
\leq \| \rho (x) \|$.  Hence (b) holds.

Finally, given (a) and (b) we see that whenever $\rho (x) = 0$, we have $\rho(x^*x) = 0$ and (b) implies that $\rho (E(x^*x)) = 0$.  Thus $E (x^*x) = 0$ by (a), and since $E$ is a faithful conditional expectation, this implies $x = 0$.  Thus $\rho$ is injective.
\end{proof}

\begin{corollary} \label{CK-for-CP-cor}
Let $X$ be a $C^*$-correspondence over a $C^*$-algebra $A$, and suppose that $X$ satisfies Condition~(S).  
If $(\psi, \pi) : (X,A) \to B$ is a covariant representation with the property that $\pi$ is injective, then the induced $*$-homomorphism $\rho_{(\psi, \pi)} : \mathcal{O}_X \to B$ is injective.
\end{corollary}

\section{Simplicity of Cuntz-Pimsner algebras} \label{simplicity-sec}

All ideals in $C^*$-algebras that we consider will be closed and two-sided.  Let $X$ be a $C^*$-correspondence over $A$.  For any ideal $I$ of $A$, we define $$IX := \clspan \{ ax : a \in I \text{ and } x \in X \}$$ and $$XI := \clspan \{ xa : a \in I \text{ and } x \in X \}.$$  We also define $$X_I := \{ x \in X : \langle y, x \rangle \in I \text{ for all $y \in X$} \}.$$  

\begin{lemma}
Let $X$ be a $C^*$-correspondence over $A$.  For any ideal $I$ of $A$ we have $XI = X_I$.
\end{lemma}

\begin{proof}
Since $X_I$ has a right Hilbert $I$-module structure, by the Cohen-Hewitt factorization theorem we have $X_I= X_I I$, which is a subset of $XI$. 

For the reverse inclusion, let $\xi \in XI := \clspan \{ xa : a \in I \text{ and } x \in X \}$.  Then $\xi = \lim_{n \to \infty} \sum_{i=1}^{N_n} x_i^n a_i^n$ with each $x_i^n \in X$ and each $a_i^n \in I$.  For any $y \in X$ we have 
$$\langle y, \xi \rangle =
\langle y,  \lim_{n \to \infty} \sum_{i=1}^{N_n}  x_i^n  a_i^n \rangle =
\lim_{n \to \infty} \langle y,   \sum_{i=1}^{N_n}  x_i^n  a_i^n \rangle =
\lim_{n \to \infty} \sum_{i=1}^{N_n} \langle y, x_i^n \rangle a_i^n \in I.$$
Thus $\xi \in X_I$, and $XI \subseteq X_I$.  
\end{proof}

Let $X$ be a $C^*$-correspondence over a $C^*$-algebra $A$.  If $I$ is an ideal in $A$, we say $I$ is \emph{hereditary} if $\phi(I) X \subseteq X I$.  We also say $I$ is \emph{saturated} if for all $a \in J_X$, $\phi(a) X \subseteq XI$ implies $a \in I$.  We call an ideal \emph{invariant} when it is both saturated and hereditary.

The following lemma is well-known among researchers studying Cuntz-Pimsner algebras.  It can be extracted as a special case of \cite[Proposition 8.2]{Kat2}, but doing so requires understanding of definitions and preliminary results in \cite{Kat2}.  For the convenience of the reader, and the benefit of the Cuntz-Pimsner literature, we provide a direct and self-contained proof for our special case.

\begin{lemma} \label{pullback-to-invariant-lem}
Let $X$ be a $C^*$-correspondence over a $C^*$--algebra $A$, and suppose that $(\psi_X, \pi_A) : (X,A) \to \mathcal{O}_X$ is a universal covariant representation.  If $\mathcal{I}$ is an ideal in $\mathcal{O}_X$, then $\pi_A^{-1} (\mathcal{I})$ is an invariant ideal in $A$.
\end{lemma}

\begin{proof}
Let $I := \pi_A^{-1} (\mathcal{I})$.  The fact that $I$ is an ideal in $A$ follows from the fact $\pi_A$ is a $*$-homomorphism.  If $a \in I$ and $x \in X$, then for any $y \in X$ we have
$$\pi_A ( \langle y, \phi(a) x \rangle ) = \psi_X (y)^* \psi_X (ax) = \psi_X (y)^* \pi_A(a) \psi_X (x) \in \mathcal{I}.$$
Thus $\phi(a) x \in X_I = XI$.  Hence $\phi(I)X \subseteq XI$, and $I$ is hereditary.

In addition, let $a \in J_X$ with $\phi(a) X \subseteq XI$.  Since 
\begin{align*}
\pi_A^{(1)} (\K (X) ) &=  \pi_A^{(1)} ( \clspan \{ \theta_{x,y} : x,y \in X \} ) \\
&=  \clspan \{\pi_A^{(1)} ( \theta_{x,y}) : x,y \in X \} =
\clspan \{ \psi_X (x) \psi_X(y)^*: x,y \in X \}
\end{align*}
we may choose an approximate unit $\{ e_\lambda \}_{\lambda \in \Lambda}$ for $\pi^{(1)} ( \K(X))$ with $$e_\lambda := \sum_{i=1}^{N_\lambda} \psi_X(x_i^\lambda) \psi_X(y_i^\lambda)^* \qquad \text{ for each $\lambda \in \Lambda$.}$$  Furthermore, since $a \in J_X$, we conclude $\pi_A( a) = \pi^{(1)} (\phi(a)) \in \pi_A ( \K (X))$, and
$$
\pi_A(a) = \lim_\lambda e_\lambda \pi_A(a) e_\lambda.
$$
Because $\phi(a) X \subseteq XI = X_I$, for any $x, y \in X$ we have $\phi(a)x \in X_I$ and $\langle y, \phi(a) x \rangle \in I$.  Thus for all $x, y \in X$,
$$\psi_X(y)^* \pi_A (a) \psi_X(x) = \pi_A ( \langle y, \phi(a)x \rangle ) \in \mathcal{I}$$
and since $\mathcal{I}$ is an ideal, it follows that $e_\lambda \pi_A(a) e_\lambda \in \mathcal{I}$ for all $\lambda \in \Lambda$.  Hence $\pi_A(a) = \lim_\lambda e_\lambda \pi_A(a) e_\lambda \in \mathcal{I}$ , so that $a \in I$ and $I$ is saturated.  Therefore $I$ is invariant.
\end{proof}

\begin{theorem}[Sufficient Conditions for Simplicity --- the Injective Left Action Case]  \label{CP-simplicity-thm}
Let $X$ be a $C^*$-correspondence over a $C^*$--algebra $A$.  If $X$ satisfies Condition~(S) and $A$ has no invariant ideals other than $A$ and $\{ 0 \}$, then $\mathcal{O}_X$ is simple.  
\end{theorem}

\begin{proof}
Let $(\psi_X, \pi_A) : (X,A) \to \mathcal{O}_X$ be a universal covariant representation.  Suppose $\mathcal{I}$ is an ideal in $\mathcal{O}_X$ with $\mathcal{I} \neq \mathcal{O}_X$.  Let $I := \pi_A^{-1}(\mathcal{I})$.  By Lemma~\ref{pullback-to-invariant-lem}, $I$ is an invariant ideal of $A$.  Hence our hypothesis implies $I = A$ or $I = \{ 0 \}$.  Since $\mathcal{I} \neq \mathcal{O}_X$, we conclude $I \neq A$ and therefore $I = \{ 0 \}$.  Define $Q_{\mathcal{I}} : \mathcal{O}_X \to \mathcal{O}_X/ \mathcal{I}$ to be the quotient map $Q_{\mathcal{I}} (x) = x + \mathcal{I}$.  Then for any $a \in A$ we have
$$Q_\mathcal{I} (\pi_A(a)) = 0 \iff \pi_A(a)+\mathcal{I} = 0 + \mathcal{I} \iff \pi_A(a) \in \mathcal{I} \iff a \in I \iff a = 0.$$
Hence $Q_\mathcal{I} |_{\pi_A(A)}$ is injective, and Theorem~\ref{CKUT-for-CP-algebras-thm} implies $Q_{\mathcal{I}} : \mathcal{O}_X \to \mathcal{O}_X/ \mathcal{I}$ is injective.  But the only way this can occur is if $\mathcal{I} = \{ 0 \}$.  Hence $\mathcal{O}_X$ is simple.
\end{proof}

\begin{remark}
Theorem~\ref{CP-simplicity-thm} gives sufficient conditions for simplicity of a Cuntz-Pimsner algebra.  However, at this time we do not know if Condition~(S) is necessary for simplicity.  If it is, then Condition~(S) would be an appropriate generalization of Condition~(L) to $C^*$-correspondences with injective left action.
\end{remark}

\begin{remark}[Comparison of Theorem~\ref{CP-simplicity-thm} and Schweizer's Simplicity Theorem]
Let $X$ be a $C^*$-correspondence over a $C^*$-algebra $A$.  Schweizer's Simplicity Theorem states that under the hypotheses that $X$ is full, $A$ is unital, and the left action $\phi_X : A \to \Li(X)$ is injective, the Cuntz-Pimsner algebra $\mathcal{O}_X$ is simple if and only if $X$ is nonperiodic and $A$ has no hereditary ideal other than $\{ 0 \}$ and $A$.    

Our result in Theorem~\ref{CP-simplicity-thm} shows that $X$ satisfying Condition~(S) and $A$ having no invariant ideals implies the Cuntz-Pimsner algebra $\mathcal{O}_X$ is simple.  To compare Theorem~\ref{CP-simplicity-thm} to Schweizer's Simplicity Theorem, let us first consider the difference between the following two hypotheses:
\begin{itemize}
\item[(1)]  The only ideals of $A$ that are invariant (i.e., both hereditary and saturated) are $\{ 0 \}$ and $A$.
\item[(2)] The only ideals of $A$ that are hereditary are $\{ 0 \}$ and $A$.
\end{itemize} 
One trivially has $(2) \implies (1)$.  Thus Theorem~\ref{CP-simplicity-thm}, which is the result
\begin{center}
Condition~(S) and (1) $\implies$ $\mathcal{O}_X$ is simple,
\end{center}
is stronger than (and hence implies) the weaker result
\begin{center}
Condition~(S) and (2) $\implies$ $\mathcal{O}_X$ is simple.
\end{center}
Schweizer shows that under his hypotheses, simplicity of $\mathcal{O}_X$ is equivalent to ``Nonperiodic and (2)".  However, when one assumes Schweizer's hypotheses (i.e., $X$ is full, $A$ is unital, and the left action $\phi_X : A \to \Li(X)$ is injective), then (1) and (2) are equivalent.  (To see this in the special case of graph algebras, note that Schweizer's three hypotheses are equivalent to the graph having no sources, finitely many vertices, and no sinks, respectively.)

We have chosen to state Theorem~\ref{CP-simplicity-thm} using the hypothesis (2) for two reasons: First, this formulation produces a stronger sufficiency result.  Second, the corresponding simplicity results for $C^*$-algebras of graphs, topological graphs, and topological quivers is ``$C^*(E)$ is simple $\iff$ $E$ satisfies Condition~(L) and $E^0$ has no nontrivial subsets that are both saturated and hereditary".  Our formulation using (2) not only establishes parallelism between these results, but also anticipates that if a future characterization of simplicity of a Cuntz-Pimsner algebra is found, it will most likely be of the form ``some condition on $X$" together with ``$A$ has no saturated and hereditary ideals other than $\{ 0 \}$ and $A$".

It remains to compare Schweizer's nonperiodic condition and our Condition~(S).  This comparison is more involved, and it is the subject of  Section~\ref{examination-sec} in this paper.
\end{remark}

We conclude this section by examining the situation when the left action is not necessarily injective.

\subsection{Simplicity when the left action is not necessarily injective}

Since Condition~(S) requires the left action of a $C^*$-correspondence to be injective, our result in Theorem~\ref{CP-simplicity-thm} only provides sufficient conditions for simplicity of a Cuntz-Pimsner algebra when the $C^*$-correspondence has injective left action.  However, by combining Theorem~\ref{CP-simplicity-thm} with \cite[Corollary~9.16]{KM}, we are able to derive sufficient conditions for simplicity of the Cuntz-Pimsner algebra of any $C^*$-correspondence, which we establish in Theorem~\ref{simple-sufficient-general-thm}.  Before doing so, we recall the definition of the dilation of a $C^*$-correspondence.

\begin{definition} \label{dilation-def}
Suppose that $X$ is a $C^*$-correspondence over $A$, that $\mathcal{O}_X$ is the Cuntz-Pimsner algebra of $X$, and that $(\psi_X, \pi_A) : (X,A) \to \mathcal{O}_X$ is a universal covariant representation.  Also let $\gamma : \T \to \operatorname{Aut} \mathcal{O}_X$ denote the gauge action of $\mathcal{O}_X$, and let $\overline{A} := \mathcal{O}_X^\gamma$ denote the fixed-point algebra of $\gamma$.  It follows from \cite[Remark~10.7]{Kat2} and \cite[Proposition~5.7]{Kat} that
$$\{ \xi \in \mathcal{O}_X : \gamma_z (\xi) = z\xi \text{ for all $z \in \T$} \} = \overline{\operatorname{span}} (\psi_X(X) \mathcal{O}_X^\gamma ).$$
We define $\overline{X} := \{ \xi \in \mathcal{O}_X : \gamma_z (\xi) = z\xi \text{ for all $z \in \T$} \} = \overline{\operatorname{span}} (\psi_X(X) \mathcal{O}_X^\gamma )$, and observe that $\overline{X}$ is a Hilbert $\overline{A}$-bimodule with inner products
$$ {}_{\overline{X}}\langle \xi, \eta \rangle := \xi \eta^* \quad \text{ and } \quad  \langle \xi, \eta \rangle_{\overline{X}} := \xi^* \eta \qquad \text{ for all $\xi, \eta \in \overline{X}$.} $$
The Hilbert bimodule $(\overline{X}, \overline{A})$ is called the \emph{dilation} of the $C^*$-correspondence $(X,A)$.  As described in \cite[Proposition~10.8]{Kat2}, the natural embeddings of $\overline{X}$ and $\overline{A}$ into $\mathcal{O}_X$ constitute a covariant representation that provides a gauge-equivariant isomorphism $\mathcal{O}_{\overline{X}} \cong \mathcal{O}_X$.  It is also straightforward to show that if the left action of $\overline{A}$ on $\overline{X}$ is injective, then the left action of $A$ on $X$ is injective.
\end{definition}

\begin{remark} \label{noninjective-Schweizer-condition-rem}
If $X$ is a $C^*$-correspondence over $A$, a hypothesis used by Schweizer and others in the study of Cuntz-Pimsner algebras is that $\overline{X}^{\otimes n} \not \cong \overline{A}$ for all $n \in \N$.  It is useful to observe that if $X^{\otimes n} \cong A$ for some $n \in \N$, then $X \otimes X^{\otimes n-1} \cong X^{\otimes n-1} \otimes X  \cong A$ and $X$ is an isomorphism in the category that has $C^*$-algebras as objects and $C^*$-correspondences as morphisms with tensor product as composition.  In this situation \cite[Proposition~2.6]{EKQR} implies that $X$ is an imprimitivity bimodule, and consequently the left action of $A$ on $X$ is injective.  Therefore, using the facts of Remark~\ref{dilation-def}, we see that if $(\overline{X}, \overline{A})$ denotes the dilation of $(X,A)$ the following holds:  If the left action of $A$ on $X$ is not injective, then the left action of $\overline{A}$ on $\overline{X}$ is not injective, and  $\overline{X}^{\otimes n} \not \cong \overline{A}$ for all $n \in \N$.
\end{remark}

The facts from Definition~\ref{dilation-def} and Remark~\ref{noninjective-Schweizer-condition-rem} allow us to obtain the following theorem, which provides sufficient conditions for the simplicity the Cuntz-Pimsner algebra of any $C^*$-correspondence.  Notably, the conditions are divided into two cases, depending on whether or not the left action of $A$ on $X$ is injective.

\begin{theorem}[Sufficient Conditions for Simplicity of Cuntz-Pimsner Algebras]  \label{simple-sufficient-general-thm}
Let $X$ be a $C^*$-correspondence over $A$.
\begin{itemize}
\item[(1)]  If the left action of $A$ on $X$ is not injective, then then $\mathcal{O}_X$ is simple if and only if $A$ has no $X$-invariant ideals other than $A$ and $\{ 0 \}$.
\item[(2)]  If the left action of $A$ on $X$ is injective, $X$ satisfies Condition~(S), and $A$ has no $X$-invariant ideals other than $A$ and $\{0 \}$, then $\mathcal{O}_X$ is simple.
\end{itemize}
\end{theorem}
\begin{proof}
For (1), let $(\overline{X}, \overline{A})$ denote the dilation of $(X,A)$ as described in Definition~\ref{dilation-def}.  The fact that the left action of $A$ on $X$ is not injective implies the left action of $\overline{A}$ on $\overline{X}$ is not injective, and hence $\overline{X}^{\otimes n} \not \cong \overline{A}$ for all $n \in \N$ (see Remark~\ref{noninjective-Schweizer-condition-rem}).  It follows from \cite[Corollary~9.16]{KM} that $\overline{A}$ has no $\overline{X}$-invariant ideals other than $\overline{A}$ and $\{ 0 \}$ if and only if $\mathcal{O}_{\overline{X}} \cong \mathcal{O}_X$ is simple.  Moreover, \cite[Theorem~8.6]{Kat2} shows that the gauge-invariant ideals of a Cuntz-Pimsner algebra are in one-to-one correspondence with $O$-pairs (defined in  \cite[Definition~5.6 and Definition~5.13]{Kat2}) of the coefficient algebra.  Consequently, the coefficient algebra of a $C^*$-correspondence has no invariant ideals other than itself and $\{ 0 \}$ if and only if the associated Cuntz-Pimsner algebra has no gauge-invariant ideals other than itself and $\{ 0 \}$ (or, using popular terminology: the coefficient algebra is minimal if and only if the Cuntz-Pimnser algebra is gauge-simple).  Putting these results together gives
\begin{align*}
& \text{$A$ has no $X$-invariant ideals other than $A$ and $\{ 0 \}$} \\
\iff &\text{$\mathcal{O}_X$ has no gauge-invariant ideals other than $\mathcal{O}_X$ and $\{ 0 \}$ (\cite[Theorem~8.6]{Kat2})} \\
\iff &\text{$\mathcal{O}_{\overline{X}}$ has no gauge-invariant ideals other than $\mathcal{O}_{\overline{X}}$ and $\{ 0 \}$ } \\
&\hspace{1.8in} \text{(since $\mathcal{O}_X$ is equivariantly isomorphic to $\mathcal{O}_{\overline{X}}$)}  \\
\iff & \text{$\overline{A}$ has no $\overline{X}$-invariant ideals other than $\overline{A}$ and $\{ 0 \}$ \  (\cite[Theorem~8.6]{Kat2})} \\
\iff & \text{$\mathcal{O}_{\overline{X}}$ is simple (\cite[Corollary~9.16]{KM})} \\
\iff & \text{$\mathcal{O}_X$ is simple (since $\mathcal{O}_X$ is isomorphic $\mathcal{O}_{\overline{X}}$).}
\end{align*}
Finally, (2) follows from Theorem~\ref{CP-simplicity-thm}.
\end{proof}

\section{Topological Quivers and Condition~(L)} \label{L-iff-S-sec}

\begin{definition}
A topological quiver is a quintuple $\Q = (E^0, E^1, r, s, \lambda)$ consisting of a second countable locally compact Hausdorff space $E^0$ (whose elements are called vertices), a second countable locally compact Hausdorff space $E^1$ (whose elements are called edges), a continuous open map $r : E^1 \to E^0$, a continuous map $s : E^1 \to E^0$, and a family of Radon measures $\lambda = \{ \lambda_v \}_{v \in E^0}$ on $E^1$ satisfying the following two conditions:
\begin{itemize}
\item[(1)] $\operatorname{supp} \lambda_v = r^{-1}(v)$ for all $v \in E^0$, and
\item[(2)] $v \mapsto \int_{E^1} \xi (\alpha) \, d\lambda_v(\alpha)$ is an element of $C_c(E^0)$ for all $\xi \in C_c(E^1)$.
\end{itemize}
\end{definition}

\begin{definition}[The quiver $C^*$-correspondence]
If $\Q = (E^0, E^1, r, s, \lambda)$ is a quiver, we let $A := C_0(E^0)$ and define an $A$-valued inner product on $C_c(E^1)$ by 
$$ \langle \xi, \eta \rangle (v) := \int_{r^{-1}(v)} \overline{\xi (\alpha)} \eta (\alpha) \, d\lambda_v (\alpha) \qquad \text{ for $v \in E^0$ and $\xi, \eta \in C_c(E^1)$}.$$
We let $X(\Q)$ denote the closure of $C_c(E^1)$ with respect to the norm arising from this inner product.  We define a right action of $A$ on $X(\Q)$ by 
$$ (\xi \cdot f) (\alpha) := \xi (\alpha) f (r(\alpha)) \qquad \text{ for $\alpha \in E^1$, $\xi \in C_c(E^1)$, and $f \in C_0(E^0)$}$$
and extending this action to all of $X(\Q)$.  We define a left action $\phi : A \to \Li (X(\Q))$ by 
$$ (\phi(f) \xi) (\alpha) := f(s(\alpha)) \xi (\alpha) \qquad \text{ for $\alpha \in E^1$, $\xi \in C_c(E^1)$, and $f \in C_0(E^0)$} $$
and extending to all of $X(\Q)$.  With this inner product and these actions $X(\Q)$ is a $C^*$-correspondence that we call the \emph{quiver $C^*$-correspondence} of $\Q$.
\end{definition}

\begin{remark} \label{quiver-vector-function-rem}
Although $X(\Q) := \overline{C_c(E^1)}$, as described in \cite[Remark~2.2]{MT2} the elements of $X(\Q)$ may be viewed as the continuous sections of a continuous field of Hilbert spaces.  In particular, if $\xi \in X(\Q)$, then for any $v \in E^0$ the vector $\xi$ determines a function in $L^2 (r^{-1}(v), \lambda_v)$.  Thus when working with $X(\Q)$, the evaluation of the inner product via the integral $\int_{r^{-1}(v)} \overline{\xi (\alpha)} \eta (\alpha) \, d\lambda_v (\alpha)$ can be performed for all vectors $\xi, \eta \in X(\Q)$ and not only those in $C_c(E^1)$.
\end{remark}

\begin{definition}[The quiver $C^*$-algebra]
If $\Q = (E^0, E^1, r, s, \lambda)$ is a quiver, let $X(\Q)$ denote the quiver $C^*$-correspondence of $\Q$.  The \emph{quiver $C^*$-algebra} $C^*(\Q)$ is defined to be the Cuntz-Pimsner algebra $\mathcal{O}_{X(\Q)}$.
\end{definition}

Both topological graph $C^*$-algebras and graph $C^*$-algebras are special cases of topological quiver $C^*$-algebras.  Topological graphs and their $C^*$-algebras are the special case when $r$ is a local homeomorphism and each $\lambda_v$ is counting measure on $r^{-1}(v)$.  Graphs and their $C^*$-algebras are the special case when $E^0$ and $E^1$ both have the discrete topology and each $\lambda_v$ is counting measure on $r^{-1}(v)$.

If $\Q = (E^0, E^1, r, s, \lambda)$ is a topological quiver, we define $E^0_\textnormal{sinks} := E^0 \setminus \overline{s(E^1)}$.  An element of $E^0_\textnormal{sinks}$ is called a \emph{sink}, and we say $\Q$ has no sinks if $E^0_\textnormal{sinks}$ is empty (or, equivalently, if the image of $s : E^1 \to E^0$ is dense in $E^0$).  It is shown in \cite[Proposition~3.15(1)]{MT2} that $\ker \phi_{X(\Q)} = C_0( E^0_\textnormal{sinks})$, and thus the left action $\phi_{X(\Q)}$ is injective if and only if $\Q$ has no sinks.

If $\Q = (E^0, E^1, r, s, \lambda)$ is a topological quiver, a \emph{path} in $\Q$ is a finite sequence of edges $\alpha := e_1 \ldots e_n$ with $r(e_i) = s(e_{i+1})$ for all $1 \leq i \leq n-1$.  We say that the path $\alpha$ has length $| \alpha | := n$, and we let $E^n$ denote the set of paths in $\Q$ of length $n$.  We may extend the maps $r$ and $s$ to $r^n : E^n \to E^0$ and $s^n : E^n \to E^0$ by setting $r^n(\alpha) := r(e_n)$ and $s^n(\alpha) := s(e_1)$.  When no confusion arises, we shall omit the superscript and simply write $r$ (respectively, $s$) for $r^n$ (respectively, $s^n$).

We endow $E^n$ with the topology it inherits as a subspace of the cartesian product $E^1 \times \ldots \times E^1$.  Since $s$ is continuous and $r$ is continuous and open, we see that $s^n$ is continuous and $r^n$ is continuous and open.  We also define a Radon measure $\lambda_v^n$ on $E^n$ with support in $(r^n)^{-1}(v)$ inductively as follows: For $\xi : E^n \to \C$ we let
\begin{align*}
&\int_{(r^n)^{-1}(v)} \xi (e_1 \ldots e_n) \, d\lambda_v^n (e_1 \ldots e_n) \\
&:= \int_{r^{-1} (s(e_2))} \ldots   \int_{r^{-1} (s(e_n))}\int_{r^{-1} (v)}  \xi (e_1 \ldots e_n) \ d\lambda_{s(e_2)} (e_1) \ldots d\lambda_{s(e_n)} (e_{n-1}) d \lambda_{v} (e_n).
\end{align*}
Since the support of $\lambda_v$ is $r^{-1}(v)$ for all $v \in E^0$, it follows that the support of $\lambda_v^n$ is $(r^n)^{-1}(v)$.

If $\Q = (E^0, E^1, r, s, \lambda)$ is a topological quiver, then for any $n \in \N$ we have that $\Q^n := (E^0, E^n, r, s, \lambda^n)$ is a topological quiver as well.  Furthermore, if $X(\Q)$ is the quiver $C^*$-correspondence for $\Q$ over $A := C_0(E^0)$, then $X(\Q)^{\otimes n}$ is isomorphic to $X(\Q^n)$ (cf. \cite[Section~6]{MT2}).  Moreover, we see that for $\xi, \eta \in X(\Q)^{\otimes n}$ we have that $\langle \xi, \eta \rangle \in C_0(E^0)$ is the function with 
$$\langle \xi, \eta \rangle (v) := \int_{(r^n)^{-1}(v)} \overline{\xi (\alpha)} \eta (\alpha) \, d\lambda_v^n (\alpha) \qquad \text{ for $v \in E^0$.}$$

A path $\alpha = e_1 \ldots e_n$ is called a \emph{cycle} if $s(e_1) = r(e_n)$, and we call the vertex $s(e_1) = r(e_n)$ the \emph{base point} of the cycle.  An \emph{exit} for a cycle $\alpha = e_1 \ldots e_n$ is an edge $f \in E^1$ with the property that $s(f) = s(e_i)$ and $f \neq e_i$ for some $1 \leq i \leq n$.  (Note: In earlier literature, cycles we often called ``loops".  In modern terminology, the typical convention is to use the term ``loop" only for cycles of length one.)  A cycle $\alpha = e_1 \ldots e_n$ is called \emph{simple} when $s(e_1) \neq r(e_i)$ for $1 \leq i \leq n-1$.

$ $

\noindent \textbf{Condition~(L):}  If $\Q$ is a topological quiver, we say that $\Q$ satisfies Condition~(L) if the set of base points of cycles with no exits has empty interior in $E^0$.

$ $

\noindent Note that a graph satisfies Condition~(L) if and only if every cycle in the graph has an exit.

\begin{definition}
A path $\alpha := e_1 \ldots e_n \in E^n$ is \emph{returning} if $e_k = e_n$ for some $k \in \{ 1, 2, \ldots, n-1 \}$.  Otherwise $\alpha$ is said to be \emph{nonreturning}.
\end{definition}

\begin{definition}
A nonempty subset $U \subseteq E^n$ is \emph{nonreturning} if whenever $\alpha := e_1 \ldots e_n \in U$ and $\beta := f_1 \ldots f_n \in U$ we have that $\alpha_n \neq \beta_k$ for all $k \in \{ 1, \ldots, n-1 \}$.
\end{definition}

\begin{proposition} \label{L-implies-S-prop}
Let $\Q = (E^0, E^1, r, s, \lambda)$ be a topological quiver with no sinks, and let $X(\Q)$ be the associated quiver $C^*$-correspondence over $A := C_0(E^0)$.  If $\Q$ satisfies Condition~(L), then $X(\Q)$ satisfies Condition~(S).
\end{proposition}

\begin{proof}
Let $a \in A := C_0(E^0)$ with $a \geq 0$, let $n \in \N$, and let $\epsilon >0$.  Set $$V := \{ v \in E^0 : a(v) > \| a \| - \epsilon \} \subseteq E^0.$$
Since $V$ is a nonempty open subset of $E^0$, \cite[Lemma~6.11]{MT2} implies that there exists $m > n$ and a nonreturning path $\alpha \in (s^m)^{-1} (V)$.  By \cite[Lemma~6.12]{MT2} there exists an open set $U \subseteq (s^m)^{-1}(V)$ such that $\alpha \in U$ and $U$ is nonreturning.  Using the fact that $U$ is an open subset of the locally compact space $E^m$, choose a continuous compactly supported function $\zeta : E^m \to [0,1]$ with $\| \zeta \| = 1$ and with the support of $\zeta$ contained in $U$. It follows from \cite[Lemma~6.11]{MT2} that $\zeta$ is a nonreturning vector in $X(\Q)^{\otimes m}$.

Since $\langle \zeta, \zeta \rangle$ is a positive continuous function in $C_0(E^0)$ with $ \| \langle \zeta, \zeta \rangle \| = \| \zeta \|^2 = 1$, by the extreme value theorem there exists $v_0 \in E^0$ with $\langle \zeta, \zeta \rangle (v_0) = 1$.  In addition, we have
$$ \langle \zeta, a \zeta \rangle (v_0) = \int_{ (r^m)^{-1} (v_0) } \overline{ \zeta (\alpha) } a (s (\alpha)) \zeta (\alpha) \, d\lambda_{v_0}^n (\alpha).$$
Since $\zeta$ is supported in $U$, whenever $\zeta (\alpha) \neq 0$ we have $\alpha \in U$ and $s(\alpha) \in V$ so that $a(s(\alpha)) > \| a \| - \epsilon$.  Thus
\begin{align*}
\langle \zeta, a \zeta \rangle (v_0) &= \int_{ (r^m)^{-1} (v_0) } \overline{ \zeta (\alpha) } a (s (\alpha)) \zeta (\alpha) \, d\lambda_{v_0}^n (\alpha) \\
&= \int_{ (r^m)^{-1} (v_0) } a(s (\alpha)) \, | \zeta (\alpha)|^2 \, d\lambda_{v_0}^n (\alpha) \\
&\geq \int_{ (r^m)^{-1} (v_0) } ( \| a \| - \epsilon) \, | \zeta (\alpha)|^2 \, d\lambda_{v_0}^n (\alpha) \\
& = ( \| a \| - \epsilon) \int_{ (r^m)^{-1} (v_0) }  | \zeta (\alpha)|^2 \, d\lambda_{v_0}^n (\alpha) \\
& = ( \| a \| - \epsilon) \int_{ (r^m)^{-1} (v_0) }  \overline{ \zeta (\alpha) } \zeta (\alpha) \, d\lambda_{v_0}^n (\alpha) \\
&= ( \| a \| - \epsilon) \, \langle \zeta, \zeta \rangle (v_0) \\
&= ( \| a \| - \epsilon) \cdot 1 \\
&= ( \| a \| - \epsilon). 
\end{align*}
Hence $\| \langle \zeta, a \zeta \rangle \| \geq \| a \| - \epsilon$, and $X(\Q)$ satisfies Condition~(S).
\end{proof}

\begin{proposition} \label{S-implies-L-prop}
Let $\Q = (E^0, E^1, r, s, \lambda)$ be a topological quiver, and let $X(\Q)$ be the associated quiver $C^*$-correspondence over $A := C_0(E^0)$.  If $X(\Q)$ satisfies Condition~(S), then $\Q$ satisfies Condition~(L).
\end{proposition}

\begin{proof}
We shall establish the contrapositive:  If $\Q$ does not satisfy Condition~(L), then $X(\Q)$ does not satisfy Condition~(S).  

Suppose that $\Q$ does not satisfy Condition~(L).  Then there exists a nonempty open subset $U \subseteq E^0$ such that every vertex in $U$ is the base point of a cycle with no exits.  Since $E^0$ is locally compact, there exists a nonempty precompact subset $V \subseteq U$ with $\overline{V} \subseteq U$.  Because every point in $U$ is the base point of a cycle with no exits, we see that $\overline{V} \subseteq U \subseteq \bigcup_{n=1}^\infty r ( (s^n)^{-1} (U))$.  In addition, since $s^n$ is continuous and $r$ is open, for each $n \in \N$ the set $r ( (s^n)^{-1} (U))$  is open.  Therefore $\{ r ( (s^n)^{-1} (U)) : n \in \N \}$ is an open cover of $\overline{V}$, and by the compactness of $\overline{V}$ there exists $N \in \N$ such that $\overline{V} \subseteq \bigcup_{n=1}^N r ( (s^n)^{-1} (U))$.  Thus $V \subseteq \bigcup_{n=1}^N r ( (s^n)^{-1} (U))$, and hence every vertex in $V$ is the base point of a cycle with no exits having length at most $N$.  

To prove that $X(\Q)$ does not satisfy Condition~(S), we shall show that there exists $a \in A := C_0(E^0)$ with $a \geq 0$ and there exists $\epsilon > 0$ such that for all $m > N$ when $\zeta \in X(\Q)^{\otimes m}$ with $ \| \zeta \| = 1$ and $\| \langle \zeta, a \zeta \rangle \|  > \| a \| - \epsilon$, then $\zeta$ is not a nonreturning vector.

Using the fact that $E^0$ is locally compact, we choose a nonzero compactly supported $a \in C_0(E^0)$ with $a \geq 0$ and the support of $a$ contained in $V$.  Also let $\epsilon := \| a \| / 2$.  Suppose $m > N$ and $\zeta \in X(\Q)^{\otimes m}$ with $\| \zeta \| =1$ and $\| \langle \zeta, a \zeta \rangle \| > \| a \| - \epsilon = \| a \| / 2$.  We shall show that $\zeta$ is not a nonreturning vector.

As described in Remark~\ref{quiver-vector-function-rem}, $\zeta$ may be viewed as a complex-valued function.  Choose $v_0 \in E^0$ with $\langle \zeta, a \zeta \rangle (v_0) >  \| a \| / 2$.  Since 
\begin{equation} \label{integral-no-cycles-eq}
\langle \zeta, a \zeta \rangle (v_0) = \int_{ (r^m)^{-1} (v_0) } \overline{ \zeta (\alpha) } a (s (\alpha)) \zeta (\alpha) \, d\lambda_{v_0}^m (\alpha),
\end{equation}
and $a$ has support contained in $V$, there exists a path $\alpha_0 \in E^m$ with $s(\alpha_0) \in V$ and $r(\alpha_0) = v_0$.  Because every vertex in $V$ is the base point of a cycle with no exits of length at most $N$, and since $N < m$, we may write 
$$\alpha_0 = \underbrace{\beta \beta \ldots \beta}_{\text{$k$ times}} \gamma$$ 
where $\beta$ is the unique simple cycle based at $s(\alpha_0) \in W$ and $\gamma$ is a subpath of $\beta$.   However, we see that $\alpha_0$ is then the unique path of length $m$ with source in $V$ and range equal to $v_0$.  Hence 
$$\int_{ (r^m)^{-1} (v_0) } \overline{ \zeta (\alpha) } a (s (\alpha)) \zeta (\alpha) \, d\lambda_{v_0}^m (\alpha) = a (s(\alpha_0) )\, | \zeta (\alpha_0) |^2 \, \lambda_{v_0}^m ( \{ \alpha_0 \} )$$
and \eqref{integral-no-cycles-eq} implies
\begin{equation} \label{product-and-measure-nonzero-eq}
a (s(\alpha_0) )\, | \zeta (\alpha_0) |^2 \, \lambda_{v_0}^m ( \{ \alpha_0 \} )  = \langle \zeta, a \zeta \rangle (v_0) > \| a \| / 2 > 0.
\end{equation}
It follows that $\lambda_{v_0}^m ( \{ \alpha_0 \} ) > 0$.  Let $n := | \beta|$, and note that $n \leq N < m$.  Since $s(\beta) = s(\alpha_0) \in V$, we see that $\beta \in (s^n)^{-1} (V)$.  Because $E^n$ is locally compact, there exists a continuous compactly supported function $\xi_\beta : E^n \to [0, \infty)$ with $\xi_\beta (\beta) = a(s(\alpha_0))$ and the support of $\xi_\beta$ contained in the open set $(s^n)^{-1} (V)$.

Let $(\psi_X, \pi_A) : (X(\Q),A) \to \mathcal{O}_{X(\Q)} = C^*(\Q)$ denote a universal covariant representation of ${}_AX(\Q)_A$ into $C^*(\Q)$.  It follows from \cite[Lemma~6.4]{MT2} that
\begin{equation} \label{product-eta-eq}
\psi_X^{\otimes m} (\zeta)^* \psi_X^{\otimes n} (\xi_\beta) \psi_X^{\otimes m} ( \zeta) = \psi_X^{\otimes n} ( \eta),
\end{equation}
where $\eta : E^n \to \C$ is defined by
$$ \eta (\alpha) := \int_{(r^m)^{-1} (s(\alpha))} \overline{ \zeta (e_1 \ldots e_m)} \xi_\beta (e_1 \ldots e_n) \zeta (e_{n+1} \ldots e_m \alpha) \, d \lambda_{s(\alpha)}^m (e_1 \ldots e_m).$$
Recalling that $\gamma$ is a subpath of the cycle $\beta$, we may write $\beta = \gamma \mu$, where $\mu$ is a subpath of $\beta$ with $s(\mu) = r(\gamma)$ and $r (\mu) = s(\beta) = s(\gamma)$.  Note that $\mu \gamma$ is a cycle of length $n$, which is based at $s(\mu) = r(\gamma)$.  

As noted before, $\alpha_0 = \beta \beta \ldots \beta \gamma$ is the unique path of length $m$ with source in $V$ and range equal to $v_0$.  Hence if we take $\mu \gamma \in E^n$, then $\alpha_0$ is also the unique path of length $m$ with source in $V$ and range equal to $r(\mu \gamma) = r(\gamma) = v_0$.  Recalling that $\xi_\beta$ has support contained in $(s^n)^{-1} (V)$ (and hence $\xi_\beta$ vanishes on any path of length $n$ with source outside $V$), we conclude
\begin{align*}
\eta (\mu \gamma) &= \overline{\zeta(\alpha_0)} \ \xi_\beta (\beta) \ \zeta ( (\underbrace{\beta \beta \ldots \beta}_{\text{$k-1$ times}} \gamma) (\mu \gamma) ) \ \lambda_{v_0}^m ( \{ \alpha_0 \} ) \\
&= \overline{\zeta(\alpha_0)} \ a(s(\alpha_0)) \ \zeta ( (\underbrace{\beta \beta \ldots \beta}_{\text{$k-1$ times}} (\gamma \mu ) \gamma ) \ \lambda_{v_0}^m ( \{ \alpha_0 \} ) \\
&= \overline{\zeta(\alpha_0)} \ a(s(\alpha_0)) \ \zeta ( (\underbrace{\beta \beta \ldots \beta}_{\text{$k$ times}}  \gamma ) \ \lambda_{v_0}^m ( \{ \alpha_0 \} ) \\
&= \overline{\zeta(\alpha_0)} \ a(s(\alpha_0)) \ \zeta ( \alpha_0 ) \ \lambda_{v_0}^m ( \{ \alpha_0 \} ) \\
&=  a(s(\alpha_0)) \ | \zeta(\alpha_0)|^2  \ \lambda_{v_0}^m ( \{ \alpha_0 \} ) \\
&> 0
\end{align*}
where the last line follows from \eqref{product-and-measure-nonzero-eq}.

Thus $\eta \neq 0$.  Since $(\psi_X, \pi_A)$ is a universal covariant representation of $X(\Q)$, both $\psi_X^{\otimes n}$ and $\pi_A$ are injective, and hence $\psi_X^{\otimes n} (\eta) \neq 0$.  It follows from \eqref{product-eta-eq} that $\psi^{\otimes m}_X (\zeta)^* \psi^{\otimes n}_X (\xi_\beta) \psi^{\otimes m}_X ( \zeta) \neq 0$.  Thus $\zeta$ is not a nonreturning vector.  Hence $X(\Q)$ does not satisfy Condition~(S).
\end{proof}

\begin{theorem}
Let $\Q = (E^0, E^1, r, s, \lambda)$ be a topological quiver with no sinks, and let $X(\Q)$ be the associated quiver $C^*$-correspondence over $A := C_0(E^0)$.  Then $\Q$ satisfies Condition~(L) if and only if $X(\Q)$ satisfies Condition~(S).
\end{theorem}

\begin{proof}
The result follows from Proposition~\ref{L-implies-S-prop} and Proposition~\ref{S-implies-L-prop}.
\end{proof}

\section{An Examination of Schweizer's Simplicity Result} \label{examination-sec}

If $X$ and $Y$ are $C^*$-correspondences over a $C^*$-algebra $A$, a \emph{$C^*$-correspondence isomorphism} from $X$ to $Y$ is a bijective map $\Phi: X \to Y$  such that 
\begin{enumerate}
\item[(i)] $\Phi(a\cdot x)=a\cdot\Phi(x)$, and 
\item[(ii)] $\langle \Phi(x), \Phi(x') \rangle = \langle x,x' \rangle$,
\end{enumerate}
for all $x,x'\in X$  and $a\in A.$  We say $X$ and $Y$ are \emph{isomorphic} (or, for emphasis, \emph{isomorphic as $C^*$-correspondences}) when such as isomorphism exists.  Also note that any $C^*$-correspondence isomorphism is automatically linear and preserves the right action (i.e., $\Phi (xa) = \Phi(x)a$ for $x \in X$ and $a \in A$) due to (ii).

\begin{definition}
Let $X$ be a $C^*$-correspondence over the $C^*$-algebra $A$.  We say $X$ is \emph{full} if $$\clspan \{ \langle x, y \rangle : x,y \in X \} = A.$$
\end{definition}

\begin{remark}
Let $A$ be a $C^*$-algebra. Then $A$ is naturally a $C^*$-correspondence over itself with the bimodule structure given by multiplication and inner product is defined by $\langle a,b\rangle = a^*b$.  This $C^*$-correspondence is called the \emph{identity $C^*$-correspondence on $A$}, and it is used in the definition of Schweizer's ``nonperiodic" condition. 
\end{remark}

\begin{definition}[Nonperiodic Condition]
Let $X$ be a $C^*$-correspondence over the $C^*$-algebra $A$.  We say $X$ is \emph{nonperiodic} if for all $n \in \N \cup \{ 0 \}$,  $X^{\otimes n}  \cong A$ (as $C^*$-correspondences) implies $n=0$.  We say that a $C^*$-correspondence is \emph{periodic} if it is not nonperiodic.
\end{definition}

\begin{theorem} \textnormal{(Schweizer's Simplicity Theorem \cite[Theorem~3.9]{Sch}.)}
Let $X$ be a $C^*$-correspondence over a $C^*$-algebra $A$, and suppose that $X$ is full, $A$ is unital, and the left action $\phi_X : A \to \mathcal{L} (X)$ is injective.  Then $\mathcal{O}_X$ is simple if and only if $X$ is nonperiodic and $A$ has no hereditary ideals other than $\{ 0 \}$ and $A$.
\end{theorem}

\begin{remark}
Although the hypothesis of the injectivity of the left action is not mentioned in the statement of \cite[Theorem~3.9]{Sch}, it is used to apply \cite[Theorem~3.8]{Sch} in its proof.  Our statement above of Schweizer's theorem makes this hypothesis explicit.
\end{remark}

\begin{definition}
If $X$ is a $C^*$-correspondence over a $C^*$-algebra $A$, we say $X$ is \emph{nondegenerate} when $\clspan \{ a \cdot x : a \in A \text{ and } x \in X \} = X$.
\end{definition}

\begin{definition}
Let $E=(E^0,E^1,r,s)$ be a directed graph. The graph correspondence is the nondegenerate $C^*$-correspondence $X(E)$ over $C_0(E^0)$ defined by 
\[ X(E)=\{\xi: E^1\rightarrow\mathbb{C}: \text{ the map } v\mapsto \sum_{r(e)=v} |\xi(e)|^2 \text{ is in } C_0(E^0)\},\]
with module actions and $C_0(E^0)$-valued inner product given by 
$$(a\cdot \xi)(e)= a(s(e))\xi(e), \ \ (\xi\cdot b)(e)=\xi(e)b(r(e)), \ \ \text{ and} \ \ \langle \xi,\eta\rangle(v)=\sum_{r(e)=v}\overline{\xi(e)}\eta(e)$$
for $a,b\in C_0(E^0),$ $\xi, \eta\in X(E),$ $e\in E^1,$ and $v\in E^0.$ 
\end{definition}

\begin{definition}
If $E = (E^0, E^1, r, s)$ is a graph, then for any $n \in \N$ we let $E^n$ denote the set of paths in $E$ of length $n$.  We also let $E^{\times n} := (E^0, E^n, r, s)$ denote the graph whose vertices are the vertices of $E$ and whose edges are the paths of length $n$ of $E$.
\end{definition}

\begin{remark} It follows from \cite[Proposition 3.3]{CDE} that if $E$ is a graph, then $X(E)^{\otimes n} \cong X(E^{\times n})$ as $C^*$-correspondences for all $n \in \N$.
\end{remark}

\begin{example}
Let $E$ be the graph
$$
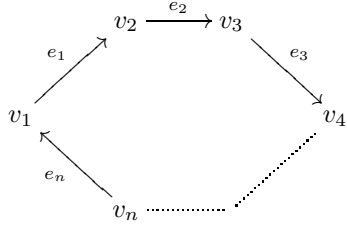
\xymatrix{
& v_2 \ar[r]^{e_2} & v_3 \ar[rd]^{e_3} & \\
v_1 \ar[ru]^{e_1} & & & v_4 \\
& v_n \ar[lu]^{e_n} \ar@{.}[r]& \ar@{.}[ru] & \\
}
$$
consisting of a simple cycle with $n$ vertices for some $n \in \N$.   Note that $A := C(E^0) \cong \C^n$.  Since each vertex of $E$ is the base point of a cycle of length $n$, we see that $E^{\times n}$ is the graph

$ $

$$
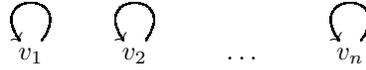
\xymatrix{
v_1 \ar@(ur,ul) & v_2 \ar@(ur,ul) & \ldots & v_n \ar@(ur,ul) 
}
$$
so that $X(E^{\times n}) \cong  \C^n$ with the left and right multiplication occurring componentwise.  Thus $X(E)^{\otimes n} \cong X(E^{\times n}) \cong \C^n \cong A$ as $C^*$-correspondences over $A$, and $X(E)$ is periodic.
\end{example}

\begin{example} \label{lcm-periodic-ex}
Let $E$ be the graph
$$
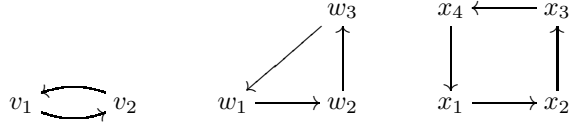
\xymatrix{
& & &  w_3 \ar[dl] & x_4 \ar[d] & x_3 \ar[l] \\
v_1 \ar@/_/[r] & v_2  \ar@/_/[l] &  w_1 \ar[r] & w_2 \ar[u] & x_1 \ar[r] & x_2 \ar[u]
}
$$
consisting of the disjoint union of three simple cycles of lengths 2, 3, and 4, respectively.  If we take the least common multiple of these lengths, $\operatorname{lcm} (2,3,4) = 12$, then we see that each vertex is the base point of a cycle of length 12, and hence $E^{\times 12}$ is the graph

$ $

$$
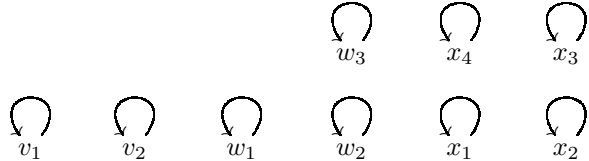
\xymatrix{
& & &  w_3  \ar@(ur,ul) & x_4  \ar@(ur,ul) & x_3  \ar@(ur,ul)  \\
v_1 \ar@(ur,ul) & v_2   \ar@(ur,ul) &  w_1  \ar@(ur,ul) & w_2  \ar@(ur,ul)  & x_1 \ar@(ur,ul) & x_2  \ar@(ur,ul) 
}
$$
In particular, since there are 9 vertices,  we have $X(E)^{\otimes 12} \cong X(E^{\times 12}) \cong \C^9 \cong A$ as $C^*$-correspondences over $A$, and thus $X(E)$ is periodic.

\end{example}

\begin{proposition}
Let $E$ be a graph with a finite number of vertices, no sinks, and no sources.  (Note that $E$ is not required to have a finite number of edges.)  Then the following are equivalent:
\begin{itemize}
\item[(1)] $X(E)$ is nonperiodic.
\item[(2)] $E$ is not a disjoint union of simple cycles.
\end{itemize}

\end{proposition}

\begin{proof}
We shall establish the result by proving the contrapositives:  $X(E)$ is periodic if and only if $E$ is a disjoint union of cycles.

If $E$ is disjoint union of cycles, then since $E$ has a finite number of vertices there are a finite number of cycles in this disjoint union.  Let $n$ be the least common multiple of the lengths of these cycles.  Then $E^{\times n}$ is a graph with a single cycle of length one at each vertex.  Hence we have the $C^*$-correspondence isomorphisms  $X(E)^{\otimes n} \cong X(E^{\times n}) \cong A$ (cf. Example~\ref{lcm-periodic-ex}).  Thus $X(E)$ is periodic.

Conversely, suppose that $X(E)$ is periodic.  Then there exists $n \in \N$ such that $X(E)^{\otimes n} \cong A$ as $C^*$-correspondences.  For convenience of notation, label the vertices as $E^0 = \{1, \ldots, k \}$.  Then $C(E^0) \cong \C^k$, and we have the $C^*$-correspondence isomorphisms $X(E^{\times n}) \cong X(E)^{\otimes n} \cong \C^k \cong A$.  Recalling that $X(E^{\times n}) := C(E^n)$, we conclude there exists a $\C^k$-linear isomorphism $h : C(E^n) \to C(E^0) \cong \C^k$.  For each $\alpha \in E^n$, let $\delta_\alpha : E^n \to \C$ be the function with $\delta_\alpha (\alpha) = 1$ and $\delta_\alpha (\beta) = 0$ when $\beta \neq \alpha$.   Also, for each $1 \leq i \leq k$, let $e_i$ denote the standard basis element of $\C^k$ having 1 in entry $i$ and $0$ in all other entries.  Suppose $h( \delta_\alpha) = (z_1, \ldots, z_k)$.  Then
\begin{align*}
e_{r(\alpha)} &= \langle \delta_\alpha, \delta_\alpha \rangle = \langle h(\delta_\alpha), h(\delta_\alpha) \rangle = (z_1, \ldots, z_k)^* (z_1, \ldots, z_k) \\
&= (\overline{z_1}, \ldots, \overline{z_k}) (z_1, \ldots, z_k) = ( |z_1|^2, \ldots, |z_k|^2).
\end{align*}
Consequently $z_i = 0$ for $i \neq r(\alpha)$, and $z_{r(\alpha)} \neq 0$.  Thus $h(\delta_\alpha) =  z_{r(\alpha)} e_{r(\alpha)}$.  Hence
$$ z_{r(\alpha)} e_{r(\alpha)} = h(\delta_\alpha) = h( e_{s(\alpha)} \cdot \delta_\alpha ) = e_{s(\alpha)} h(\delta_\alpha) = e_{s(\alpha)} \cdot (z_1, \ldots, z_k) = z_{s(\alpha)} e_{s(\alpha)}.
$$
Since $z_{r(\alpha)} \neq 0$, the above equation implies $s(\alpha) = r(\alpha)$.  Because $\alpha$ is an arbitrary element of $E^n$, we conclude every element of $E^n$ is a cycle.  Moreover, since $C(E^n) \cong C(E^0) \cong \C^k$, we deduce $E^n$ has $k$ elements.  Thus $E^n = \{ \alpha_1, \ldots, \alpha_k \}$, where $\alpha_i$ is a cycle of length $n$ for each $1 \leq i \leq k$.  Therefore $E$ is a disjoint union of simple cycles.
\end{proof}

\begin{definition}
Let $E = (E^0, E^1, r, s)$ be a graph.  We say $E$ is \emph{weakly connected} if for all $v,w \in E^0$ there exists a sequence of vertices $v_1, \ldots, v_n \in E^0$ with $v_1 = v$, $v_n = w$, and for all $1 \leq i \leq n$ there is either an edge in $E$ from $v_i$ to $v_{i+1}$ or there is an edge in $E$ from $v_{i+1}$ to $v_i$.  We say $E$ is strongly connected if for all $v,w \in E^0$ there is a path in $E$ from $v$ to $w$.
\end{definition}

\begin{corollary} \label{weakly-connected-nonperiodic-cor}
Let $E$ be a weakly connected graph with a finite number of vertices, no sinks, and no sources.  (Note that $E$ is not required to have a finite number of edges.)  Then the following are equivalent:
\begin{itemize}
\item[(1)] $X(E)$ is nonperiodic.
\item[(2)] $E$ is not a simple cycle.
\end{itemize}
\end{corollary}

\begin{corollary} \label{strongly-connected-nonperiodic-cor}
Let $E$ be a strongly connected graph with a finite number of vertices.  (Note that $E$ is not required to have a finite number of edges.)  Then the following are equivalent:
\begin{itemize}
\item[(1)] $X(E)$ is nonperiodic.
\item[(2)] $E$ is not a simple cycle.
\item[(3)] $E$ satisfies Condition~(L).
\end{itemize}
\end{corollary}

\begin{proof}
The equivalence $(1) \iff (2)$ follows from Corollary~\ref{weakly-connected-nonperiodic-cor}.  

To see $(2) \implies (3)$, suppose $E$ is not a simple cycle.  If $\alpha$ is any cycle in $E$, then since $E$ is not a simple cycle, there exists a vertex $w \in E^0$ that is not on $\alpha$.  Since $E$ is strongly connected, there exists a path from $s(\alpha)$ to $w$, and some edge of this path must be an exit for $\alpha$.  Hence $E$ satisfies Condition~(L), and $(3)$ holds,

To see $(3) \implies (2)$, simply note that if every cycle in $E$ has an exit, then $E$ cannot be a simple cycle.
\end{proof}

\begin{example} \label{weakly-connected-ex}
If $E$ is the graph

$$
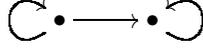
\xymatrix{
\bullet \ar[r] \ar@(dl,ul) & \bullet \ar@(dr,ur)
}
$$

$ $

\noindent then $E$ is weakly connected and does not satisfy Condition~(L).  In addition, for any $n \in \N$, we see that there are paths of length $n$ in $E$ that are not cycles.  Hence, the correspondence $X(E^{\times n})$ is not isomorphic to the correspondence $C_0(E^0)$, so that the correspondence $X(E)^{\otimes n}$ is not isomorphic to the identity correspondence on $A$.  Consequently, $X(E)$ is nonperiodic.  
\end{example}

\begin{example} \label{weakly-connected-no-invariant-ideals-ex}
If $E$ is the graph

$$
\xymatrix{
\bullet \ar[r] & \bullet \ar@(dr,ur)
}
$$

$ $

\noindent then $E$ is weakly connected and does not satisfy Condition~(L).  In addition, for any $n \in \N$, we see that there are paths of length $n$ in $E$ that are not cycles.  Hence, the correspondence $X(E^{\times n})$ is not isomorphic to the correspondence $C_0(E^0)$, so that the correspondence $X(E)^{\otimes n}$ is not isomorphic to the identity correspondence on  $A$.  Consequently, $X(E)$ is nonperiodic.  Furthermore, $E$ has no nontrivial saturated hereditary subsets, and hence $A \cong \C^2$ has no invariant ideals other than $\{ 0 \}$ and $A$.  Since $E$ does not satisfy Condition~(L), we conclude $\mathcal{O}_{X(E)} \cong C^*(E)$ is not simple.  (In fact, $C^*(E) \cong M_2(C(\T))$, so $C^*(E)$ is Morita equivalent to $C(\T)$, and $C^*(E)$ has uncountably many ideals.)  Note that since $E$ has a source, the correspondence $X(E)$ is not full, and thus the hypotheses of Schweizer's Theorem are not met.
\end{example}

\begin{example} \label{infinite-no-invariant-ideals-ex}
If $E$ is the graph

$$
\xymatrix{
\cdots \ar[r] & \bullet \ar[r] & \bullet \ar[r] & \bullet \ar[r] & \bullet \ar@(dr,ur)
}
$$

$ $

\noindent then $E$ is weakly connected, has no sinks and no sources, and does not satisfy Condition~(L).  In addition, for any $n \in \N$, we see that there are paths of length $n$ in $E$ that are not cycles.  Hence, the correspondence $X(E^{\times n})$ is not isomorphic to the correspondence $C_0(E^0)$, so that the correspondence $X(E)^{\otimes n}$ is not isomorphic to the identity correspondence on $A$.  Consequently, $X(E)$ is nonperiodic. Furthermore, $E$ has no nontrivial saturated hereditary subsets, and hence $A = C_0(E^0)$ has no invariant ideals other than $\{ 0 \}$ and $A$.  Since $E$ does not satisfy Condition~(L), $\mathcal{O}_{X(E)} \cong C^*(E)$ is not simple.  (In fact, $C^*(E) \cong C(\T) \otimes \K$, so $C^*(E)$ is Morita equivalent to $C(\T)$, and $C^*(E)$ has uncountably many ideals.)
\end{example}

\begin{example}
For each $k \in \N$ let $C_k$ be the graph that is a simple cycle of length $k$.  If $E := C_1 \sqcup C_2 \sqcup \ldots$, then for any $n \in \N$, we see that $E$ contains a path of length $n$ that is not a cycle.  Hence, the correspondence $X(E^{\times n})$ is not isomorphic to the correspondence $C_0(E^0)$, so that the correspondence $X(E)^{\otimes n}$ is not isomorphic to the identity correspondence on $A$.  Consequently, $X(E)$ is nonperiodic.  Thus $X(E)$ is nonperiodic even though $E$ is a disjoint union of simple cycles.
\end{example}

\begin{example} \label{infinite-strongly-connected-ex}
Let $E$ be a strongly connected graph with an infinite number of vertices.  If $\alpha$ is any cycle in $E$, then there exists a vertex $w \in E^0$ that is not on $\alpha$.  Since $E$ is strongly connected, there exists a path from $s(\alpha)$ to $w$, and some edge of this path must be an exit for $\alpha$.  Hence $E$ satisfies Condition~(L).

In addition, since $E$ is strongly connected, for any $n \in \N$ there exists a path of length $n$ in $E$ that is not a cycle.  Hence, the correspondence $X(E^{\times n})$ is not isomorphic to the correspondence $C_0(E^0)$, so that the correspondence $X(E)^{\otimes n}$ is not isomorphic to the identity correspondence on $A$.  Consequently, $X(E)$ is nonperiodic.  

Therefore if $E$ is any strongly connected graph with infinitely many vertices, then $E$ satisfies Condition~(L) and $X(E)$ is nonperiodic.
\end{example}

\begin{remark} \label{conclusions-comparison-rem}

\noindent For a graph $E$, let $X(E)$ denote the graph $C^*$-correspondence over $A := C_0(E^0)$.  Our results and examples from this section can be summarized by the following points:

\begin{itemize}

\item Corollary~\ref{strongly-connected-nonperiodic-cor} states that if $E$ is strongly connected with finitely many vertices, $X(E)$ being nonperiodic is equivalent to $E$ satisfying Condition~(L). 

\item  Example~\ref{weakly-connected-ex} shows that there exists a weakly connected finite graph with no sinks and no sources such that $X(E)$ is nonperiodic, but $E$ does not satisfy Condition~(L).  This shows that, even for finite graphs with no sinks and no sources, $X(E)$ being non-periodic does not imply that $E$ satisfies Condition~(L).  Moreover, if $E$ is the graph of Example~\ref{weakly-connected-ex}, then $C^*(E)$ has proper non-zero ideals that do not contain any vertex projects.  Thus, even in the special case of graph $C^*$-algebras (indeed, even when the graph is weakly connected, finite, has no sinks, and has no sources), the condition of nonperiodic is insufficient to obtain a Cuntz-Krieger Uniqueness Theorem.

\item Example~\ref{weakly-connected-no-invariant-ideals-ex} shows that there exists a weakly connected finite graph $E$ with no sinks (but having a source) for which $X(E)$ is nonperiodic and $A$ has no invariant ideals other than $\{ 0 \}$ and $A$, but $E$ does not satisfy Condition~(L).  Since $E$ does not satisfy Condition~(L), the graph $C^*$-algebra $C^*(E)$ is not simple.  This shows that the conditions ``$X$ is nonperiodic and $A$ has no hereditary ideals other than $\{ 0 \}$ and $A$" are insufficient to imply simplicity of $\mathcal{O}_X$ when $X$ is not full.  Thus a characterization of simplicity for Cuntz-Pimsner algebras of $C^*$-correspondences that are not full (even under the hypotheses that $A$ unital and the left action of $A$ on $X$ is injective) will require conditions other than ``$X$ is nonperiodic and $A$ has no hereditary ideals other than $\{ 0 \}$ and $A$".

\item Example~\ref{infinite-no-invariant-ideals-ex} shows that there exists a weakly connected graph $E$ with no sinks and no sources (but having a countably infinite number of vertices) for which $X(E)$ is nonperiodic and $A$ has no invariant ideals other than $\{ 0 \}$ and $A$, but $E$ does not satisfy Condition~(L).  Since $E$ does not satisfy Condition~(L), the graph $C^*$-algebra $C^*(E)$ is not simple.  This shows that the conditions ``$X$ nonperiodic and $A$ has no hereditary ideals other than $\{ 0 \}$ and $A$" are insufficient to imply simplicity of $\mathcal{O}_X$ when $A$ is not unital.  Thus a characterization of simplicity for Cuntz-Pimsner algebras of $C^*$-correspondences over nonunital $C^*$-algebras (even under the hypotheses that $X$ is full and the left action of $A$ on $X$ is injective) will require conditions other than ``$X$ nonperiodic and $A$ has no hereditary ideals other than $\{ 0 \}$ and $A$".

\item Finally, Example~\ref{infinite-strongly-connected-ex} shows that for strongly connected graphs with an infinite number of vertices, nonperiodic and Condition~(L) are both automatic.  Moreover, in this situation the fact that the graph is strongly connected implies the vertex set has no hereditary subsets that are proper and nonempty, and hence the associated $C^*$-algebra is automatically simple.  
\end{itemize}
\end{remark}

These observations show that while the notion of nonperiodic is useful for characterizing simplicity under the hypotheses of Schweizer's Simplicity Theorem, it is insufficient for characterizing simplicity more generally.  Moreover, even when Schweizer's hypotheses hold, the condition of nonperiodic is insufficient for obtaining a Cuntz-Krieger Uniqueness Theorem.  Indeed, we see this even in the special case of graph $C^*$-algebras, where Condition~(L) is a necessary and sufficient condition for a Cuntz-Krieger Uniqueness Theorem to hold.

This gives additional credibility for the usefulness of the notion of Condition~(S) that we introduced in this paper.  First of all, under the hypothesis that the left acton is injective, Condition~(S) coincides with Condition~(L) for graphs, topological graphs, and topological quivers, suggesting it is a reasonable generalization of Condition~(L) for general $C^*$-correspondences.  Second, unlike the nonperiodic condition, we have been able to use Condition~(S) to obtain a Cuntz-Krieger Uniqueness Theorem for Cuntz-Pimsner algebras of $C^*$-correspondence with injective left action.  Third, our Cuntz-Krieger Uniqueness theorem allows us to obtain sufficient conditions for simplicity of Cuntz-Pimsner algebras in terms of Condition~(S).

\section{Concluding Remarks}

We conclude with two questions that are important for extending the work in this paper.

\medskip

\noindent \textbf{Question 1:} Let $X$ be a $C^*$-correspondence over $A$ with injective left action.  If $\mathcal{O}_X$ is simple, is it necessary that $X$ satisfies Condition~(S)?

\medskip

\noindent When $X = X(\Q)$ is the $C^*$-correspondence of a topological quiver $\Q$, we know that $\mathcal{O}_{X(E)} = C^*(\Q)$ simple implies $\Q$ satisfies Condition~(L), which implies $X(\Q)$ satisfies Condition~(S).  Thus we have an affirmative answer to Question~(1) for $C^*$-correspondences of topological quivers, topological graphs, and graphs.  If an affirmative answer to Question~(1) is obtained in general,  then Condition~(S) is precisely the correct generalization of Condition~(L) for $C^*$-correspondences with injective left action.  Moreover, an affirmative answer to Question~(1) would provide a converse to Theorem~\ref{simple-sufficient-general-thm}(2), giving necessary and sufficient conditions for the simplicity of the Cuntz-Krieger algebra of any $C^*$-correspondence.

\medskip

\noindent \textbf{Question 2:} How can Condition~(S) be appropriately modified to obtain a Cuntz-Krieger Uniqueness Theorem for general $C^*$-correspondences whose left action is not necessarily injective?

\medskip

\noindent As noted in Proposition~\ref{Condition-S-implies-injective-prop}, a $C^*$-correspondence satisfying Condition~(S) necessarily has injective left action.  It would be desirable to remove the hypothesis of injective left action and obtain an appropriate condition for general $C^*$-correspondences.  While there are techniques for extending results for $C^*$-correspondences with injective left actions to the general situation (e.g., via the method of ``adding tails" described in \cite{MT}), translating Condition~(S) to the non-injective setting has been complicated and elusive, and we have been unable to apply these techniques in any fruitful way.  Currently, we are uncertain how to modify Condition~(S) when the left action is not injective.


\begin{thebibliography}{99}


\bibitem{CDE}
T.M.~Carlsen, A.~Dor-On and S.~Eilers, \emph{Shift equivalences through the lens of Cuntz-Krieger algebras}, arXiv.2011.10329v5.

\bibitem{CKO}
T.M.~Carlsen, B.K.~Kw\'asniewski, and E.~Ortega, \emph{Topological freeness for $C^*$- correspondences}, J.~Math.~Anal.~Appl. \textbf{473} (2019), 749--785.

\bibitem{COP}
T.M.~Carlsen, E.~Ortega, and E.~Pardo, \emph{Simple Cuntz-Pimsner rings}, J.~Algebra \textbf{371} (2012), 367--390.

\bibitem{EKQR}
S.~Echterhoff, S.~Kaliszewski, J.~Quigg, and I.~Raeburn, \emph{Naturality and induced representations}, Bull. Austral. Math. Soc. \textbf{61} (2000), 415--438.


\bibitem{FMR}
N.~Fowler, P.S.~Muhly, and I.~Raeburn, \emph{Representations of Cuntz-Pimsner algebras}, Indiana Univ. Math. J. \textbf{52} (2003), 569--605.

\bibitem{FR}
N.~Fowler and I.~Raeburn, \emph{The Toeplitz algebra of a Hilbert bimodule}, Indiana Univ. Math. J. \textbf{48}
(1999), 155--181.

\bibitem{KPW}
T.~Kajiwara, C.~Pinzari, Y.~Watatani, \emph{Ideal structure and simplicity of the $C^*$-algebras
generated by Hilbert bimodules}, J. Funct. Anal. \textbf{159} (1998), 295--322.

\bibitem{Kat3}
T.~Katsura, \emph{A class of $C^*$-algebras generalizing both graph algebras and homeomorphism $C^*$-algebras. I. Fundamental results}, Trans. Amer. Math. Soc. \textbf{356} (2004), 4287--4322. 

\bibitem{Kat}
T.~Katsura, \emph{On $C^*$-algebras associated with $C^*$-correspondences}, J. Funct. Anal. {\bf 217} (2004), 366--401.

\bibitem{Kat2}
T.~Katsura, \emph{Ideal structure of $C^*$-algebras associated with $C^*$-correspondences}, Pacific J. Math. \textbf{230} (2007), 107--145.

\bibitem{KM}
B.K.~Kwa\'sniewski and R.~Meyer,  \emph{Aperiodicity, topological freeness and pure outerness: from group actions to Fell bundles}, Studia Math.  \textbf{241} (2018), 257--303.

\bibitem{Lan}
E.~C.~Lance, Hilbert $C^*$-modules. A toolkit for operator algebraists. London Mathematical
Society Lecture Note Series, \textbf{210}. Cambridge University Press, Cambridge, 1995.

\bibitem{MS}
P.~S.~Muhly and B.~Solel, \emph{On the simplicity of some Cuntz-Pimsner algebras}, Math. Scand. \textbf{83} (1998), 53--73. 

\bibitem{MS2}
P.~S.~Muhly and B.~Solel, \emph{Corrigendum to:~``On the simplicity of some Cuntz-Pimsner algebras''}, Math. Scand. \textbf{91} (2002), 244--246.

\bibitem{MT}
P.~S.~Muhly and M.~Tomforde, \emph{Adding tails to $C^*$-correspondences}, Documenta Math. \textbf{9} (2004), 79--106.

\bibitem{MT2}
P.S.~Muhly and M.~Tomforde, \emph{Topological quivers}, Internat. J. Math. \textbf{16} (2005), no. 7, 693--755.

\bibitem{Pim}
M.~V.~Pimsner, \emph{A class of $C^*$-algebras generalizing both Cuntz-Krieger algebras and crossed products by $\Z$}, Fields Institute Comm. \textbf{12} (1997), 189--212.

\bibitem{Sch}
J.~Schweizer, \emph{Dilations of $C^*$-correspondences and the simplicity of Cuntz-Pimsner algebras}, J. Funct. Anal. \textbf{180} (2001), 404--425. 


\end{thebibliography}
\end{document}